\documentclass{imsart}

\RequirePackage{amsthm,amsmath}
\RequirePackage[numbers]{natbib}
\RequirePackage[colorlinks,citecolor=blue,urlcolor=blue]{hyperref}
\RequirePackage{graphicx}
\RequirePackage[resetlabels]{multibib}

\usepackage{bm,amsfonts}
\usepackage{multirow}

\startlocaldefs
\numberwithin{equation}{section}
\allowdisplaybreaks
\theoremstyle{plain}
\theoremstyle{plain}
\newtheorem{thm}{Theorem}[section]

\newtheorem{lem}[thm]{Lemma}
\newtheorem{corollary}[thm]{Corollary}
\theoremstyle{remark}

\newtheorem{assm}{Assumption}
\newtheorem{example}{Example}

\newtheorem{rem}{Remark}

\def\T{{ \mathrm{\scriptscriptstyle T} }}
\def\v{{\varepsilon}}
\def\R{\mathbb R}

\def\E{\mathbb E}
\renewcommand{\P}{\mathbb P}

\renewcommand{\hat}{\widehat}
\renewcommand{\tilde}{\widetilde}
\renewcommand{\bar}{\overline}
\newcommand{\vv}[1]{\boldsymbol{#1}}

\endlocaldefs

\begin{document}

\begin{frontmatter}
\title{Many-sample tests for the dimensionality hypothesis for large covariance matrices among groups%\support{Support information of the article.}
}
\runtitle{Many-sample dimensionality test}

\begin{aug}
    \author{\fnms{Tianxing}  \snm{Mei}\ead[label=e1]{tianxingmei@ln.edu.hk}},
    \author{\fnms{Chen} \snm{Wang}%\thanksref{t1}
    \ead[label=e2]{stacw@hku.hk}}
    \and
    \author{\fnms{Jianfeng} \snm{Yao~}\ead[label=e3]{jeffyao@cuhk.edu.cn}}
    %    \runauthor{T. Mei, C. Wang and    J. Yao}
    %\thankstext{t1}{Corresponding author.}
    \address{
      Hong Kong Institute of Business Studies, Faculty of Business\\
      Lingnan University\\
      \printead{e1}}
    \address{
      Department of Statistics and Actuarial Science \\
      The University of Hong Kong \\
      \printead{e2}}
      \address{
      School of Data Science  \\
      The Chinese University of Hong Kong  (Shenzhen)\\
      \printead{e3}}
      
\end{aug}

%\begin{aug}
%\author{\fnms{First} \snm{Author}\thanksref{t1,t2}\ead[label=e1]{first@somewhere.com}}
%\and
%\author{\fnms{Second} \snm{Author}\thanksref{t3}\ead[label=e2]{second@somewhere.com}}

%\address{Address of the First and Second authors\\
%usually few lines long\\
%\printead{e1,e2}}

%\author{\fnms{Third} \snm{Author}
%\ead[label=e3]{third@somewhere.com}
%\ead[label=u1,url]{www.foo.com}}

%\address{Address of the Third author\\
%usually few lines long\\
%usually few lines long\\
%\printead{e3}\\
%\printead{u1}}

%\thankstext{t1}{Some comment}
%\thankstext{t2}{First supporter of the project}
%\thankstext{t3}{Second supporter of the project}
%\runauthor{F. Author et al.}

%\end{aug}

\begin{abstract}
In this paper, we consider procedures for testing hypotheses on the dimension of the linear span generated by a growing number of $p\times p$ covariance matrices from independent $q$ populations.  Under a proper limiting scheme where all the parameters, $q$,
$p$, and the sample sizes from the $q$ populations, are allowed to increase to infinity, we derive the asymptotic normality of the proposed test statistics. 
The proposed test procedures show satisfactory performance in finite samples under both the null and the alternative.
We also apply the proposed many-sample dimensionality test to investigate a matrix-valued gene dataset from the Mouse Aging Project and gain some new knowledge about its covariance structures.  
\end{abstract}


\tableofcontents
\end{frontmatter}

\section{Introduction}

%High-dimensional data are nowadays found in many modern scientific fields, such as genomic study, high-frequency stock pricing, or wireless communication. Traditional statistical tools for parameter estimation or hypothesis testing show very poor performance for such high-dimensional data since their validity typically needs a large sample size $n$ with respect to the number of variables $p$. New mathematical theories have emerged in the last two decades to handle high-dimensional datasets where the dimension $p$ can have a comparable or even larger magnitude than the sample size $n$. This paper concerns a particular problem in this area of high-dimensional statistics, namely hypothesis testing on covariance matrices. 
 
Hypothesis testing of covariance matrices is one of the fundamental topics in multivariate data analysis.
Many widely used statistical methods heavily rely on the structures of covariance matrices, including principal component analysis, discriminant analysis, and clustering analysis; see \cite{anderson2003,muirhead2009aspects} for more details. Therefore, testing for structural hypotheses on covariance matrices is crucial for the applications of these statistical tools.
This topic has been extensively studied in the recent literature, along with increasingly complex hypotheses on structures; see \cite{Cai2017423,Hu20162281} for a detailed review.
% Introduce existing works with a focus on different hypotheses

There are various different types of   hypotheses considered in the rich literature, among which the following two types are of special interest. 
% One sample case
The first type cares about the uncorrelation among components of a single population $\vv{x}\in\R^p$, reflected in its covariance matrix $\vv{\Sigma}$ that the corresponding entries vanish. The following are contained in this type:
\begin{itemize}
    \item[(a)] the diagonality hypothesis: ``$\vv{\Sigma}$ is diagonal'' in \cite{Lan201576,Xu20173208,Yinqiu2021154,Bodnar20192977};
    
    \item[(b)] the block-diagonality hypothesis: ``$\vv{\Sigma} = {\rm diag}(\vv{\Sigma}_{11},\ldots,\vv{\Sigma}_{hh})$'' for non-negative defined matrices $\{\vv{\Sigma}_{jj}\}$, in \cite{Hyodo2015460,Lai2022,Marques2008726,Yamada2017305};
    
    \item[(c)] the band matrix hypothesis: ``$\vv{\Sigma} = \vv{B}_k(\vv{\Sigma})$'', where the $(r,s)$th entry of $\vv{B}_k(\vv{\Sigma})$ vanishes for any $|r-s|>k$, in \cite{Yinqiu2021154,Qiu20121285}.
\end{itemize}
Testing for the independence of several populations $\vv{x}_1,\ldots,\vv{x}_q \in \R^p$ is also included in this type by combining different populations as a whole in $\R^{pq}$ and testing the block-diagonality of the covariance matrix.
Methods to test these hypotheses are usually based on the sum of squares of estimators of the zero-entries of $\vv{\Sigma}$. 

The other type focuses on the linear dependence relationship between covariance matrices, which is of the main interest of this work. In the one-sample cases, these hypotheses concerns 
whether the matrix $\vv{\Sigma}$ falls into a given sub-space of $p\times p$ matrices, including
\begin{itemize}
    \item[(i)] the identity hypothesis: ``$\vv{\Sigma} = \vv{I}_p$'' in \cite{Ledoit20021081,Bai20093822,Chen2010808};
    
    \item[(ii)] the sphericity hypothesis: ``$\vv{\Sigma} = c\vv{I}_p$'', with $c > 0$ unspecified, in \cite{Li20162973,Coelho2010711,Ledoit20021081};

    \item[(iii)] the Toeplitz-type hypothesis: ``$\vv{\Sigma}$ is a non-negative definite Toeplitz matrix'' in \cite{Butucea2016164};

    \item[(vi)] the linearity hypothesis: ``$\vv{\Sigma} = \theta_1 \vv{A}_1 + \cdots +\theta_K\vv{A}_k$'' for a given basis of symmetric covariance matrices $\{\vv{A}_j\}$ and unknown coefficients $\{\theta_j\}$, in \cite{Zheng20193300,Zhong20171185,Klein2022,Fan2022}. 
\end{itemize}
Test procedures are constructed based on a measure of the distance between the estimator of $\vv{\Sigma}$ and the given sub-space. But, in one-sample cases, the information about the matrix subspace in the null must be pre-specified, which might bring new challenges to the choices of the basis matrices of subspace in practice.    

% Multi-sample case
Hypotheses on the linear dependence of covariance matrices are also widely studied in multi-sample tests. In this scenario, we have several populations $\vv{x}_1,\ldots,\vv{x}_q\in \R^p$ with respective covariance matrices $\vv{\Sigma}_j = {\rm Cov}(\vv{x}_j)$, $j=1,\ldots,q$. 
The following hypotheses with increasingly complicated relationships are considered:
\begin{itemize}
    \item[(I)] The equality hypothesis: ``$\vv{\Sigma}_1 = \cdots = \vv{\Sigma}_q$'', in \cite{Ahmad20152619,Ahmad2017500,Coelho2010583,Coelho2012627,Ishii201999,Li2012908};

    \item[(II)] The proportionality hypothesis: ``$\vv{\Sigma}_j = w_j\vv{\Sigma}$'', where the matrix $\vv{\Sigma}$ is unspecified and $\{w_j\}$ are non-negative numbers, in \cite{Cheng2020,Liu2013293};

    \item[(III)] The linear hypothesis: for fixed $k\in\{1,\ldots,q\}$,
    \begin{equation}\label{eq:1.1}
       a_{i1} \vv{\Sigma}_1 + \cdots + a_{iq} \vv{\Sigma}_q = \vv{O},~~~~~i=1,\ldots,k,
    \end{equation}
    where coefficients $\{a_{ij}\}$ are given, as in \cite{Bai2021701};

    \item[(IV)]  {\em The dimensionality hypothesis} $H_0$: ``$d = d_0$” against the alternative $H_1$: ``$d \geq  d_0 + 1$”, where $d$ is the dimension of the linear span of $\{\vv{\Sigma}_j\}$ and $d_0$ is a pre-specified integer $d_0$.
\end{itemize}
% Dimensionality hypothesis 
Observe that by letting $a_{i,i} =1$ and $a_{i,i+1} = -1$ in (\ref{eq:1.1}), hypothesis (I) is included in (III). 
In addition, when $\{w_j\}$ in (II) is given, the proportionality hypothesis is also covered in (III) by choosing $a_{i,i} = w_{i+1}$ and $a_{i,i+1} = -w_i$. 
Thus, the extent of the linear hypothesis is indeed broader than the previous two. 
However, the requirement for the coefficients 
$\{a_{ij}\}$ to be known restricts the applicability of (III). 
For instance,
the proportionality hypothesis will be excluded from (III) when $\{w_j\}$ in (II) is unspecified.    

In this paper, we propose the dimensionality hypothesis (IV), which, to the best of our knowledge, has not been discussed in the literature yet.
On one hand, hypothesis (IV) is more general to cover (I) and (II) and relax the requirement coefficients $\{a_{ij}\}$ in (III). 
In fact, hypotheses (I) and (II) correspond to the simplest cases on (IV), with $d_0=0$ and $d_0=1$, respectively.
In addition, (IV) implies that any $(d_0+1)$ members of $\{\vv{\Sigma}_j\}$ are linearly dependent. Thus, $\{\vv{\Sigma}_j\}$ must satisfy a certain system of equations in the form (\ref{eq:1.1}) with the coefficient matrix $\{a_{ij}\}$ having rank $d_0$. 
The advantage of (IV) is that the coefficients $\{a_{ij}\}$ do not need to be prespecified. 
On the other hand, the study of testing hypothesis (IV) is motivated by our analyzing the Mouse Aging Project data \cite{Touloumis20211309}. The data set consists of $n=40$ independent matrix-valued observations, and each observation collects the expression levels of $p=46$ genes from $q=9$ different organs. The study in \cite{Touloumis20211309} and \cite{MWY2} concludes that the gene expression levels of these $q = 9$ groups are mutually independent, but the $q$ covariance matrices of gene expression levels are not proportional to each other. In other words, the hypothesis $H_0: ``d= 1''$ is rejected. To further investigate the covariance structure of the dataset, it is natural to develop a progressive test of the dimensionality hypothesis (IV) to find the ``true'' dimension of these covariance matrices.

% Many-sample tests  
We also consider the testing problem for hypothesis (IV) under a new high-dimensional setting called {\em many-sample} setting. 
Different from the aforementioned multi-sample case where the number of different groups is fixed, 
the number $q$ of populations in the many-sample setting is allowed to grow to infinity along with the dimension $p$ and the sample sizes at a certain rate. 
The concept of many-sample tests was proposed for the first time in our previous work \cite{MWY2} when we investigated two genetics data sets: 1000 Genomes Project phase 3 dataset  \cite{100gene}, and the Mouse Aging Project data \cite{mice07}. They collect independent samples of information on a great number of genes from different populations, but the number of populations $q$ in both scenarios is comparable to the sample sizes. Traditional methods in the multi-sample tests show poor performance due to the effect caused by a large number of populations. This  inspires the study of many-sample tests.
In our work \cite{MWY2}, we considered many-sample tests for hypotheses (I) and (II) and obtained desirable results to identify equality and proportionality. As mentioned before, (I) and (II) correspond to the special cases ``$d=0$'' and ``$d=1$'' of the hypothesis (IV). Therefore, it is natural for us to continue working on the dimensionality hypothesis (IV) for any $d_0 \geq 1$. 

% Key novelty of the topic
%In this paper, we will focus on the many-sample tests for the dimensionality hypothesis (IV) as a follow-up of \cite{MWY2}. 
%The key novelty is that we consider a completely new problem to detect the ``true'' dimension of the linear span of a growing number of covariance matrices, under a high-dimensional setting where dimension, as well as sample sizes, increase to infinity along with the number of different populations.

% Our contribution
The main contributions of this paper are as follows. We propose a distance measure to characterize the dimensionality hypothesis (IV), based on determinants of principal sub-matrices of the Gram matrix generated by $\{\vv{\Sigma}_j\}$, which equals zero under the null hypothesis and increases along the deviation from the null. 
Further, we construct a generalized $U$-statistic involving the $q$ sample covariance matrices to estimate the distance. 
Under a proper limiting scheme, 
we derive the asymptotic normality for the proposed $U$- statistic under both the null and the alternative. 
The power analysis is also conducted under a special alternative where there are only a few outliers that have their covariance matrices fall outside the linear span of the rest. We provide a geometric interpretation of how outliers' orthogonal parts to the major linear span influence the power.
Our many-sample test also shows a satisfactory performance in simulation studies under both the null and the above alternative. 

% Data analysis
We also apply our procedure to detecting covariance structures of the Mouse Aging Project data. The previous study in \cite{Touloumis20211309} and \cite{MWY2} reveals that the dataset possesses $q = 9$ independent columns but the covariance matrices of these columns are not mutually proportional. To further investigate the relationships between the $q=9$ covariance matrices, we perform the developed many-sample dimensionality tests sequentially and finally conclude that the linear span of the covariance matrices has dimension $3$.
To justify the conclusion, we treat the data from another viewpoint by approximating the joint covariance matrix $\vv{\Sigma}$ of the vectorization of columns of the dataset by Kronecker products. According to the discussion in Section \ref{sec:data}, if the linear span of $\{\vv{\Sigma}_j\}$ has a dimension $d_0$, then $\vv{\Sigma}$ can be represented as sums of three Kronecker product. The simulation result shows approximation by sums of three Kronecker products is indeed better than by a single Kronecker product, which supports the finding from  our dimensionality test.

% Organization of the paper
The rest of the paper is organized as follows.  We develop our many-sample dimensionality test in Section \ref{sec:main}. Simulation results and the real data application to the Mouse Aging Project are arranged in Sections \ref{sec:simu} and \ref{sec:data}, respectively.  Proofs of theorems and lemmas are provided in Appendices \ref{app:A} --- \ref{app:E}.

\section{Main Results}\label{sec:main}
\subsection{Basic settings and assumptions}\label{ssec:settings}
Consider $q$ populations $\vv{x}_1,\ldots,\vv{x}_q$ which are $q$ random vectors in $\R^p$  with zero mean and covariance matrix $\vv{\Sigma}_i=\vv{\Sigma}_{ip}={\rm cov}(\vv{x}_i)$, $i\in [q]:=\{1,\ldots,q\}$.  

Let $\mathcal{H}_0$ be the linear subspace spanned by $\{\vv{\Sigma}_i: i \in [q]\}$. The dimension of $\mathcal{H}_0$ is denoted by $d$. 
We consider testing the following hypothesis:
\begin{equation}\label{eq:hp}
\begin{split}
    H_0:~~d = d_0~~~\hbox{versus}~~~H_1:~~d > d_0,
\end{split}
\end{equation}  
where $d_0$ is a pre-specified positive integer much smaller than $q$. The null hypothesis means that there exists a linearly independent subset $\{\vv{\Sigma}_{i_1},\ldots,\vv{\Sigma}_{i_{d_0}}\}$ such that
the other $q - d_0$ covariance matrices can be represented as their linear combinations. In particular, when $d_0 = 1$, the null hypothesis $H_0$ reduces to that of proportionality, which has been studied in \cite{MWY2}.

For each $i\in [q]$, suppose that we have a sample of size $n_i$ from the population $\vv{x}_i$, 
denoted by $\vv{x}_{i,1},\vv{x}_{i,2},\ldots,\vv{x}_{i,n_i}$. 
Samples from different populations are assumed to be mutually independent. 
Let $\|\cdot\|$ be the operator norm of matrices. The following assumptions will be used. 
\begin{assm}\label{assm:2.1}
 For $i\in [q]$, $\vv{x}_i=\vv{\Sigma}_i^{1/2} \vv{z}_i$, where $\vv{z}_i$ has i.i.d. entries with zero mean, unit variance and finite eighth moment. Moreover, the fourth moment of entries of $\vv{z}_i$ is independent of $i$ and denoted by $\nu_4$.
\end{assm}

\begin{assm}\label{assm:2.2}
 There exist two positive constants $C$ and $c$ such that 
 \[
 \sup_{p,q}\max_{1\leq i\leq q}\|\vv{\Sigma}_i\|\leq C
 ~~~~\text{and} ~~~~
 \inf_{p,q} \frac{1}{pq} \sum_{i=1}^q {\rm tr}(\vv{\Sigma}_i) \geq c.
 \] 
\end{assm}

\begin{assm}\label{assm:2.3} 
For each $i\in [q]$,  $n_i=n_i(p)\to\infty$ as $p\to\infty$ such that $c_{i,p}:=p/n_i\to c_i\in(0,\infty)$. Moreover,  there exist positive constants $c_0$ and $C_0$ such that $c_0\leq c_i\leq C_0$ for all $i\in [q]$. 
\end{assm}

\begin{assm}\label{assm:2.4} 
The number of populations  $q=q(p)\to\infty$ such that $q=o(p^2)$.
\end{assm}

\subsection{The population Gram matrix}

We introduce {\em the population Gram matrix} to characterize the dimension of population covariance matrices. 
Let $\mathcal{H}$ be the Hilbert space of $p\times p$ symmetric matrices equipped with the matrix inner product $\langle\vv{A},\vv{B}\rangle={\rm tr}(\vv{A}\vv{B}^\T)/p$. 
Denote $\mathcal{H}_0$ the subspace spanned by $\{\vv{\Sigma}_1,\ldots,\vv{\Sigma}_q\}$.
The {\em population Gram matrix} $\vv{G}=(G_{ij})_{q\times q}$ generated by $\{\vv{\Sigma}_j\}$ is defined as a non-negative definite Hermitian matrix, whose entries are given by the inner product $G_{ij} = \langle \vv{\Sigma}_i, \vv{\Sigma}_j\rangle = {\rm tr}(\vv{\Sigma}_i\vv{\Sigma}^\T_j)/p$. 
It is well known that 
\begin{itemize}
    \item[(a)] $\vv{G}$ is non-negative definite;
    \item[(b)] $\vv{G}$ is non-degenerated if and only if $\{\vv{\Sigma}_j\}$ are linearly independent;
    \item[(c)] the rank of $\vv{G}$ is identical to the dimension of $\mathcal{H}_0$.
\end{itemize}
The determinant of $\vv{G}$ is of special interest in this paper. It can be viewed as the square of the hyper-volume of the parallelotope spanned by vectors $\{\vv{\Sigma}_j\}$ in $\mathcal{H}$. Hence, when $\{\vv{\Sigma}_j\}$ are linearly dependent, the parallelotope degenerates so that ${\rm det}(\vv{G}) = 0$.

To detect the true rank of the population Gram matrix, we define the following quantity for any integer $k = 1,\ldots,q$: 
\begin{equation}\label{eq:M}
    M_p^{(k)} = \binom{q}{k}^{-1}\sum_{(i_1,\ldots,i_k)\in [q]} {\rm det}(\vv{G}(i_1,\ldots,i_k))
\end{equation}
where the summation is taken over all subsets $1\leq i_1<\cdots<i_k\leq q$,
and $\vv{G}(i_1,\ldots,i_k)$ is the principal sub-matrix in $\vv{G}$ related to index $(i_1,\ldots,i_k)$.
As mentioned before, ${\rm det}(\vv{G}(i_1,\ldots,i_k))$ is the sqaured hyper-volume of the parallelotope related to $\{\vv{\Sigma}_{i_1},\ldots,\vv{\Sigma}_{i_k}\}$. Thus, from a geometric viewpoint, $M^{(k)}_p$ describes the average scatter degree of $k$-dimensional parallelotopes in $\mathcal{H}_0$.

\begin{lem}\label{lem:2.1}
Suppose that the dimension of $\mathcal{H}_0$ is $d$. Then, for any $k\leq d$, $M^{(k)}_p$ is positive; and for any $k > d$, $M^{(k)}_p = 0$. 
\end{lem}

The lemma shows that the dimension of $\mathcal{H}_0$ is the unique critical index of $\{M^{(k)}_p:k\in[q]\}$,
by which the sequence transits from a positive value to zero. Hence, the hypothesis (\ref{eq:hp}) can be  equivalently reformulated as 
\begin{equation}\label{eq:hp1}
    H_0:~M_p^{(d_0)} >0~\hbox{and}~M_p^{(d_0+1)} =0 ~~~~~\hbox{versus}~~~~~
    H_1:~M_p^{(d_0+1)} > 0.
\end{equation}
With this observation, the problem now reduces to test whether $M_p^{(d_0+1)}$ is zero given $M_p^{(d_0)} > 0$. To this end, we now focus on constructing estimators of $M_p^{(d_0+1)}$ and $M_p^{(d_0)}$ based on the $q$ samples $\{\vv{x}_{ij}\}_{1\leq i\leq q, 1\leq j\leq n_i}$.

\subsection{The sample Gram matrix}\label{ssec:sgm}
Let  $\vv{X}_{ip}=(\vv{x}_{i,1},\ldots,\vv{x}_{i,n_i})$ be the $p\times n_i$ data matrix from the $i$th sample. The corresponding sample covariance matrix is
\begin{equation}\label{eq:scv}
\vv{S}_{i,p}=\frac{1}{n_i}\sum_{k=1}^{n_i}\vv{x}_{i,k}\vv{x}^\T_{i,k}=\frac{1}{n_i}\vv{X}_{ip}\vv{X}^\T_{ip}.
\end{equation}
In the following, we define the sample Gram matrix $\hat{\vv{G}}$ based on the sample covariance matrices $\{\vv{S}_{i,p}\}$ and propose asymptotically unbiased estimators for $M_p^{(k)}$'s.

{\em The sample Gram matrix} $\hat{\vv{G}}=(\hat{G}_{ij})$ is defined as 
\begin{equation}\label{eq:estgram}
   \hat{G}_{ij} = \begin{cases}
   \cfrac{\frac{1}{p}{\rm tr}(\vv{S}_{i,p}^2) - c_{i,p} \left(\frac{1}{p}{\rm tr}(\vv{S}_{i,p})\right)^2 - \left(\frac{c_{i,p}}{p} - \frac{c_{i,p}^2}{p^2}\right)\hat{\eta}_{4,i}}{\left(1-\frac{2c_{i,p}}{p} \right)\left(1-\frac{c_{i,p}}{p}\right)},&\hbox{when}~ i = j;\\
   \cfrac{1}{p}{\rm tr}(\vv{S}_{i,p}\vv{S}_{j,p}),&\hbox{otherwise},
   \end{cases}
\end{equation}
where
\begin{equation}\label{est:2}
    \begin{split}
        \hat{\eta}_{4,i}  = \frac{(n_i-4)!}{4p n_i!}\sum_{1\leq j_1\neq j_2\neq j_3\neq j_4\leq n_i}\bigg\{&(\vv{x}_{i,j_1}-\vv{x}_{i,j_2})^\T(\vv{x}_{i,j_1}-\vv{x}_{i,j_2})\\
        -&(\vv{x}_{i,j_3}-\vv{x}_{i,j_4})^\T(\vv{x}_{i,j_3}-\vv{x}_{i,j_4})\bigg\}^2. 
    \end{split}
\end{equation}
To explain the conclusion, first note that for $i\neq j$, $\vv{S}_{i,p}$ and $\vv{S}_{j,p}$ are independent. Therefore, 
\[
\E[\hat{G}_{ij}] = \frac{1}{p}{\rm tr}(\E(\vv{S}_{i,p})\E[\vv{S}_{j,p}]) = G_{ij},
\]
so that $\hat{G}_{ij}$ is an unbiased estimator for $G_{ij}$. When $i = j$, this independence breaks down and actually ${\rm tr}(\vv{S}_{i,p}^2)/p$ is no longer an unbiased estimator for $G_{ii} = {\rm tr}(\vv{\Sigma}_j^2)/p$. The more complex formula in (\ref{eq:estgram}) for $\hat{G}_{ii}$ is indeed due to necessary condition to ${\rm tr}(\vv{S}_{i,p}^2)/p$ in order to obtain an unbiased estimator for $G_{ii}$. 
\begin{lem}
For $1\leq i\leq j\leq q$, $\E[\hat{G}_{ij}] = G_{ij}$.
\end{lem}
\begin{proof}
The case of $i\neq j$ has been already discussed. For the more complicate case of $i=j$, according to discussion in \cite{Cheng2020,MWY2}, it holds that
\begin{equation}
\left\{
\begin{aligned}
    \E\left[\frac{1}{p}{\rm tr}(\vv{S}_{i,p}^2)\right]
    & = \frac{1}{p}{\rm tr}(\vv{\Sigma}_{i}^2) + c_{i,p} \left(\frac{1}{p}{\rm tr}(\vv{\Sigma}_i)\right)^2 
      + \frac{c_{i,p}}{p} \left\{\frac{1}{p}{\rm tr}(\vv{\Sigma}_{i}^2) + \eta_{4,i}\right\},\\
    \E\left[\left(\frac{1}{p}{\rm tr}(\vv{S}_{i,p})\right)^2\right]
    & = \left(\frac{1}{p}{\rm tr}(\vv{\Sigma}_i)\right)^2  + \frac{c_{i,p}}{p^2}\left\{2\frac{1}{p}{\rm tr}(\vv{\Sigma}_{i}^2) + \eta_{4,i}\right\},\\
    \E[\hat{\eta}_{4,i}] 
    & = \eta_{4,i} + 4\frac{1}{p}{\rm tr}(\vv{\Sigma}_{i}^2),
\end{aligned}
\right.
\end{equation}
where
   \begin{equation}\label{eq:eta}
       \eta_{4,i} = (\nu_4-3) \frac{1}{p}{\rm tr}(\mathcal{D}(\vv{\Sigma}_i)^2),
   \end{equation}
with $\mathcal{D}(\vv{A})$ the diagonal matrix made by diagonal entries of matrix $\vv{A}$. By solving the above system of equations, we obtain  $\E[\hat{G}_{ii}] = G_{ii}$.
\end{proof}

In the next lemma, we show that determinants of all principal sub-matrices from the sample Gram matrix $\hat{\vv{G}}$ are asymptotically unbiased for the counterparts from the population Gram matrix $\vv{G}$.
\begin{lem}\label{lem:detG}
Under Assumptions \ref{assm:2.1}--\ref{assm:2.4}, for any integer $k$ less than $q$, and $k$-tuple $(i_1,\ldots,i_{k}) \in [q]$, it holds that
\begin{equation}\label{eq:Edet}
    \E\left\{{\rm det}(\hat{\vv{G}}(i_1,\ldots,i_{k}))\right\} = {\rm det}(\vv{G}(i_1,\ldots,i_{k})) + O(\frac{1}{p^2}),
\end{equation}
where $\hat{G}(i_1,\ldots,i_{k})$ and $G(i_1,\ldots,i_{k})$ are the respective principal sub-matrices related to indices $(i_1,\ldots,i_{k})$ in $\hat{\vv{G}}$ and $\vv{G}$, and the remainder $O(1/p^2)$ can be uniformly bounded by $C(k)/p^2$ with $C(k)$ a positive constant dependent only on $k$.
\end{lem}

The proof of the lemma is provided in Section \ref{app:B}.

\subsection{Testing statistics and its asymptotic properties}\label{ssec:AN}

Based on Lemma \ref{lem:detG}, we propose the following estimator for $M_p^{(k)}$:
\begin{equation}\label{eq:Udef}
    \hat{M}_p^{(k)} = \binom{q}{k}^{-1} \sum_{(i_1,\ldots,i_{k})} {\rm det}(\hat{\vv{G}}(i_1,\ldots,i_{k})). 
\end{equation}
An immediate consequence of Lemma \ref{lem:detG} suggests that
\begin{equation}\label{eq:Uexp}
\begin{split}
    \E[\hat{M}_p^{(k)}] & = \binom{q}{k}^{-1} \sum_{(i_1,\ldots,i_{k})} {\rm det}(\vv{G}(i_1,\ldots,i_{k})) + O(p^{-2}) \\
    & = M_p^{(k)} + O(p^{-2}).
\end{split}
\end{equation}
Under the null hypothesis, we have $\E[\hat{M}_p^{(d_0+1)}] = O(p^{-2})$ since $M_p^{(d_0+1)} = 0$. Further, $q^{1/2} p \E[M_p^{(d_0+1)}] = O(q^{1/2}/p) = o(1)$ due to Assumption \ref{assm:2.4}. 

We derive the asymptotic normality of $\hat{M}_p^{(d_0+1)}$ under both the null and alternative cases as follows. To address the conclusion, we introduce the following notation of the matrix determinant. Let $m$ be an arbitrary positive integer, $\vv{B}_1,\ldots,\vv{B}_m$ be $p\times p$ matrices, and $\{a_{ij}\}$ be a collection of real numbers. We define the matrix determinant by
\begin{equation}
    {\rm det}\left(\begin{matrix}
        \vv{B}_1 & \vv{B}_2 & \cdots & \vv{B}_m \\
        a_{21} & a_{22} & \cdots & a_{2m}\\
        \vdots & \vdots &\ddots&\vdots\\
        a_{m1} & a_{m2} & \cdots & a_{mm}
    \end{matrix}\right) = \sum_{k=1}^m (-1)^{1+m}A_{1k}\vv{B}_k,
\end{equation}
where $A_{1k}$ is a real number, and its value equals to the determinant of the submatrix obtained by removing matrices in the $1$th row and the $k$th column. Essentially, the matrix determinant is obtained by conducting Laplace expansion according to the first row, and its value is a linear combination of matrices in the first row.

\begin{thm}\label{thm:AN}
   Suppose that $d\geq d_0$ (so that $M_p^{(d_0)} >0$). Under Assumptions \ref{assm:2.1}--\ref{assm:2.4}, we have
   \begin{equation}\label{eq:AN}
       \frac{q^{1/2}p\left(\hat{M}_p^{(d_0+1)} - M_p^{(d_0+1)}\right)}{\sigma^{(d_0)}_p} \overset{d.}{\to}\mathcal{N}(0,1),
   \end{equation}
   in which $(\sigma_p^{(d_0)})^2 = 4(d_0+1)^2\{ (M_p^{(d_0)})^2 \beta_p + r_p^{(d_0)}\}$ with
   \begin{align}
       \beta_p & = \frac{1}{q}\sum_{i=1}^q c_{i,p}^2\left\{\frac{1}{p}{\rm tr}(\vv{\Sigma}_i^2)\right\}^2 = \frac{1}{q}\sum_{i=1}^q c_{i,p}^2 G_{ii}^2, \label{eq:beta}\\
       r_p^{(d_0)} & = \frac{1}{q} \sum_{i=1}^q c_{i,p} \left\{\frac{2}{p}{\rm tr}((\vv{\Sigma}_i^{1/2} \vv{R}_i\vv{\Sigma}_i^{1/2})^2) + \frac{(\nu_4-3)}{p} {\rm tr}(\mathcal{D}(\vv{\Sigma}_i^{1/2}\vv{R}_i\vv{\Sigma}_i^{1/2})^2)\right\},
       \label{eq:R}
   \end{align}
   and
   \begin{equation}\label{eq:Ri}
       \vv{R}_i = \binom{q-1}{d_0}^{-1}\sum_{(j_1,\ldots,j_{d_0})\neq i} {\rm det}\left(\begin{array}{cccc}
       \vv{\Sigma}_i & \vv{\Sigma}_{j_1} & \cdots & \vv{\Sigma}_{j_{d_0}}\\
       G_{j_1,i} & G_{j_1,j_1} & \cdots & G_{j_1,j_{d_0}}\\
       \vdots & \vdots & \ddots & \vdots \\
       G_{j_{d_0},i} & G_{j_{d_0},j_1} & \cdots & G_{j_{d_0},j_{d_0}}
       \end{array}\right) . 
   \end{equation}
   Under $H_0$, $\vv{R}_i= \vv{O}$ for all $i$ so that $r_p^{(d_0)}$ vanishes. Consequently, under $H_0$, we have $q^{1/2} p \hat{M}_p^{(d_0+1)}/\sigma_{0,p}^{(d_0)}\overset{d.}{\to}\mathcal{N}(0,1)$, where $(\sigma_{0,p}^{(d_0)})^2 = 4(d_0+1)^2 (M_p^{(d_0)})^2 \beta_p$.
\end{thm}

The proof of the main theorem is contained in Section \ref{app:C}.

In what follows, we consider the estimation of asymptotic variance $(\sigma_{0,p}^{(d_0)})^2$. Since $(\sigma_{0,p}^{(d_0)})^2$ is a product of $\hat{M}_p^{(d_0)}$ and $\beta_p$, we first estimate these two factors respectively and then combine them to make a consistent estimation. 
\begin{thm}\label{thm:var}
Suppose that Assumptions \ref{assm:2.1} --- \ref{assm:2.4} holds. Let \begin{equation}\label{eq:hatbeta}
    \hat{\beta}_p = \frac{1}{q} \sum_{i=1}^q c_{i,p}^2 \hat{G}_{ii}^2.
\end{equation}
Then, both $\hat{M}_p^{(d_0)}-M_p^{(d_0)}$ and $\hat{\beta}_p-\beta_p$ converge to $0$ in probability as $p$ increases. Further, we define 
\begin{equation}\label{eq:hatVar0}
    (\hat{\sigma}_{p}^{(d_0)})^2 = 4(d_0+1)^2 \left(\hat{M}_p^{(d_0)}\right)^2 \hat{\beta}_p.
\end{equation}
Then, $(\hat{\sigma}_{p}^{(d_0)})^2-(\sigma_{0,p}^{(d_0)})^2$ also converges to zero in probability. Consequently, under $H_0$, $q^{1/2} p \hat{M}_p^{(d_0+1)}/\hat{\sigma}_{p}^{(d_0)}$ converges in distribution to $\mathcal{N}(0,1)$.
\end{thm}

\begin{rem}
    The test statistic $q^{1/2}p\hat{M}_0^{(d_0+1)}/\hat{\sigma}_{p}^{(d_0)}$ is scale-invariant in the sense that if we multiple samples $\{\vv{x}_{ij}\}$ by a common factor $w>0$,  the ratio $\hat{M}_p^{(d_0+1)}/\hat{\sigma}_{0,p}^{(d_0)}$ always remain unchanged.
\end{rem}

The proof of Theorem \ref{thm:var} is contained in Section \ref{app:E}.

Finally, we propose the following procedure for the hypothesis testing problem (\ref{eq:hp}): let $\alpha \in (0,1)$ be a given significance level, and $z_{\alpha}$ be the $(1-\alpha)$th quantile of $\mathcal{N}(0,1)$,  
\[
\text{reject $H_0$ if}~~~ q^{1/2}p\hat{M}_p^{(d_0+1)}/\hat{\sigma}_{p}^{(d_0)} > z_{\alpha}.
\]

\subsection{The Analysis of Power}\label{ssec:power}

Let $\mathcal{H}_k$ be the linear span of $\{\vv{\Sigma}_i: 1\leq i\leq q-k\}$, where $k$ is an integer less than $q$.
In this section, we consider the following alternative hypothesis:
\begin{equation}\label{eq:althy}
H_1:\left\{\begin{aligned}
    &\hbox{there exists some $K\geq 1$ such that}~{\rm dim}(\mathcal{H}_K) = d_0\\
    &\hbox{and for any $j=1,\ldots,K$, $\vv{\Sigma}_{q-K+j}$ does not belong to $\mathcal{H}_K$.}
    \end{aligned}\right.
\end{equation}
It means that the majority of the populations $\{\vv{\Sigma}_i:1\leq i\leq q-K\}$ satisfies the null hypothesis, that is, the dimension of the spanned subspace is $d_0$. However, there exists a minority of outliers, say, the last $K$ groups, such that their covariance matrices fall outside $\mathcal{H}_K$. Note that the smaller $K$, the less $H_1$ will deviate from
$H_0$, and the smaller power the test will have. Hence, we focus on the following
small-$K$ assumption, which, in particular, allows $K$ to be finite. 
 \begin{assm}\label{assm:2.5}
     The number of outlier groups $K= o(q)$.
 \end{assm}

The study of the power function relies on the analysis of the orthogonal decomposition of outliers with respect to $\mathcal{H}_K$. Thus, we denote $\mathcal{H}_K^\perp$ the orthogonal complement space of $\mathcal{H}_K$, that is, 
\begin{equation}
    \mathcal{H}_K^\perp = \{\vv{A}\in\mathcal{H}: \langle\vv{A},\vv{\Sigma}_j\rangle = 0, j=1,\ldots,q-K\}.
\end{equation}
The following lemma provides an explicit representation of the orthogonal decomposition of any symmetric matrix.
\begin{lem}\label{lem:decomp}
Assume that the dimension of $\mathcal{H}_K$ is $d_0$. For a given element $\vv{\Sigma}\in\mathcal{H}$, let $(\vv{\Sigma}_0,\vv{\Sigma}_\perp)$ be the unique orthogonal decomposition of $\vv{\Sigma}$: $\vv{\Sigma} = \vv{\Sigma}_0 + \vv{\Sigma}_\perp$ with $\vv{\Sigma}_0\in\mathcal{H}_K$ and $\vv{\Sigma}_\perp \in \mathcal{H}_K^\perp$. 
For any positive integer $k = 1,\ldots,q-K$, define
\begin{equation}\label{eq:Mpk}
    M_{p,K}^{(k)} = \binom{q-K}{k}^{-1}\sum_{(i_1,\ldots,i_{k})\in[q-K]}{\rm det}(\vv{G}(i_1,\ldots,i_k)) 
\end{equation}
and
\begin{equation}\label{eq:Me}
    M^{(k)}_{p,K}(\vv{\Sigma}) = \binom{q-K}{k}^{-1} \sum_{(i_1,\ldots,i_k)\in[q-K]} {\rm det}(\vv{G}(\vv{\Sigma};i_1,\ldots,i_k)),
\end{equation}
where $\vv{G}(\vv{\Sigma};i_1,\ldots,i_k)$ is the Gram matrix generated by $\{\vv{\Sigma},\vv{\Sigma}_{i_1},\ldots,\vv{\Sigma}_{i_k}\}$. Then, we have
\begin{equation}\label{eq:e0}
    M^{(d_0)}_{p,K}\vv{\Sigma}_0 = -\binom{q-K}{d_0}^{-1}\sum_{(i_1,\ldots,i_{d_0})\in[q-K]}{\rm det}\begin{pmatrix}
    \vv{O} & \vv{\Sigma}_{i_1} & \cdots &\vv{\Sigma}_{i_{d_0}} \\
    \langle\vv{\Sigma}_{i_1},\vv{\Sigma}\rangle &
    \langle\vv{\Sigma}_{i_1},\vv{\Sigma}_{i_1}\rangle&
    \cdots&
    \langle\vv{\Sigma}_{i_1},\vv{\Sigma}_{i_{d_0}}\rangle\\
    \vdots & \vdots& \ddots&\vdots\\
    \langle\vv{\Sigma}_{i_{d_0}},\vv{\Sigma}\rangle &
    \langle\vv{\Sigma}_{i_{d_0}},\vv{\Sigma}_{i_1}\rangle&
    \cdots&
    \langle\vv{\Sigma}_{i_{d_0}},\vv{\Sigma}_{i_{d_0}}\rangle
    \end{pmatrix}
\end{equation}
and
\begin{equation}\label{eq:eperp}
    M^{({d_0})}_{p,K} \vv{\Sigma}_{\perp} = \binom{q-K}{{d_0}}^{-1}\sum_{(i_1,\ldots,i_{d_0})\in[q-K]}{\rm det}\begin{pmatrix}
    \vv{\Sigma} & \vv{\Sigma}_{i_1} & \cdots &\vv{\Sigma}_{i_{d_0}} \\
    \langle\vv{\Sigma}_{i_1},\vv{\Sigma}\rangle &
    \langle\vv{\Sigma}_{i_1},\vv{\Sigma}_{i_1}\rangle&
    \cdots&
    \langle\vv{\Sigma}_{i_1},\vv{\Sigma}_{i_{d_0}}\rangle\\
    \vdots & \vdots& \ddots&\vdots\\
    \langle\vv{\Sigma}_{i_{d_0}},\vv{\Sigma}\rangle &
    \langle\vv{\Sigma}_{i_{d_0}},\vv{\Sigma}_{i_1}\rangle&
    \cdots&
    \langle\vv{\Sigma}_{i_{d_0}},\vv{\Sigma}_{i_{d_0}}\rangle
    \end{pmatrix}
\end{equation}
Consequently, it holds that
\begin{equation}\label{eq:detdecomp}
    M^{({d_0})}_{p,K}(\vv{\Sigma}) = M^{({d_0})}_{p,K}\langle\vv{\Sigma}_\perp,\vv{\Sigma}_\perp\rangle. 
\end{equation}
\end{lem}
The detailed proof of Lemma \ref{lem:decomp} is given in Section \ref{app:decomp}.

Now, we let $(\vv{\Sigma}_{0,q-K+j},\vv{\Sigma}_{\perp,q-K+j})$ be the orthogonal decomposition of the outlier $\vv{\Sigma}_{q-K+j}$ for $j=1,\ldots,K$. The following theorem reveals how the orthogonal parts $\{\vv{\Sigma}_{\perp,q-K+j}\}$ influences the power. 
\begin{thm}\label{thm:pw}
Suppose that Assumptions \ref{assm:2.1} --- \ref{assm:2.5} and the alternative hypothesis (\ref{eq:althy}) hold. Then, we have 
\[
\frac{q^{1/2}p(\hat{M}_p^{(d_0+1)}-M_p^{(d_0+1)})}{2(d_0+1)M_{p,K}^{(d_0)}\beta_p^{1/2}} \overset{d.}{\to}\mathcal{N}(0,1),
\]
in which, for $p$ large enough,
\begin{equation}\label{eq:ratio}
    \frac{M_p^{(d_0+1)}}{(d_0+1)M_{p,K}^{(d_0)}} \geq \frac{1}{q}\left\{\sum_{j=1}^K \frac{1}{p} {\rm tr}(\vv{\Sigma}_{\perp,q-K+j}^2)\right\}(1+o(1)). 
\end{equation}
Moreover, we  have $(\hat{\sigma}_p^{(d_0)})^2 -4(d_0+1)^2({M}_{p,K}^{(d_0)})^2 \beta_p$ converges to zero in probability. Consequently, with 
\begin{equation}\label{eq:outbound}
    \gamma_{p,q} := \frac{1}{2\sqrt{q}\beta_p^{1/2}}\sum_{j=1}^K {\rm tr}\left[(\vv{\Sigma}_{\perp,q-K+j})^2\right] >0,
\end{equation}
and for a given significant level $\alpha\in(0,1)$, the power function satisfies, for $p$ large enough,
\begin{equation}\label{eq:altpw}
    \begin{split}
    \P_{H_1}\left(\frac{q^{1/2}p\hat{M}_p^{(d_0+1)}}{\hat{\sigma}_{p}^{(d_0)}} > z_{\alpha}\right)
    \geq \Phi\left(\gamma_{p,q} - z_{\alpha}\right) + o(1),
    \end{split}
\end{equation}
where $\Phi$ is the distribution function of $\mathcal{N}(0,1)$ and $z_{\alpha}$ is the $(1-\alpha)$th quantile of $\mathcal{N}(0,1)$. 
In particular, if
\[
    \gamma := \liminf_{p\to\infty} \frac{1}{\beta_p^{1/2}}\sum_{j=1}^K \frac{1}{p}{\rm tr}\left[(\vv{\Sigma}_{\perp,q-K+j})^2\right] >0,
\]
the power function goes to $1$ for any significant level $\alpha\in(0,1)$. 
\end{thm}
\begin{rem}
The ratio in (\ref{eq:outbound}) is also scale-invariant and captures the deviation degree of outliers from $\mathcal{H}_K$. Indeed,
according to the definition (\ref{eq:beta}), $\beta_p$ is a weighted average of squares of $\langle\vv{\Sigma}_j,\vv{\Sigma}_j\rangle$. So, the denominator $\beta_p^{1/2}$ can be viewed as certain average of squared lengths of $\{\vv{\Sigma}_j\}$. Then, (\ref{eq:outbound}) is a fraction of the sum of squared lengths of outliers' orthogonal parts to the averaging square length of all populations. The larger the ratio, the more significant outliers tend to the orthogonal space $\mathcal{H}_K^\perp$.
\end{rem}
The detailed proof of Theorem \ref{thm:pw} is provided in Section \ref{app:D}.

In what follows, we consider two examples of the 'needle-in-a-haystack' problems, which means there is only one outlier group, i.e., $K=1$.

\begin{example}[$d_0=2$]\label{e1}
Suppose that $\{\vv{\Lambda}_1,\vv{\Lambda}_2,\vv{\Lambda}_3\}$ are three linearly independent non-negative definite matrices, i.e., ${\rm det}(\vv{G}(1,2,3))>0$. Let $\{I_j:j=1,\ldots,q\}$ be a set of integers taken randomly from $\{1,2\}$. Suppose that
\begin{equation}\label{eq:ex1}
\vv{\Sigma}_j = \begin{cases}
    \vv{\Lambda}_{I_j}, &\hbox{if}~j=1,\ldots,q-1;\\
    (1-w)\vv{\Lambda}_{I_q} + w\vv{\Lambda}_3, & \hbox{if}~j =q,
\end{cases}
\end{equation}
where $0\leq w\leq 1$. In addition, we assume that $\{1,2\}\subset\{I_j:j\leq q-1\}$ so that ${\rm dim}(\mathcal{H}_1) = 2$. (Otherwise, the space $\mathcal{H}_1$ will degenerate with dimension strictly less than $2$.)
It can be seen from (\ref{eq:ex1}) that the null hypothesis $H_0$ holds when $w =0$ and the alternative $H_1$ holds when $w>0$.

To analyze the power, we let $\vv{\Lambda}_{\perp,3}$ be the part of $\vv{\Lambda}_3$ orthogonal to the subspace ${\rm span}(\{\vv{\Lambda}_1,\vv{\Lambda}_2\})$. Then, by (\ref{eq:Mpk}), (\ref{eq:Me}) and (\ref{eq:detdecomp}), we find
\begin{align*}
M_p^{(3)} & = \binom{q}{3}^{-1}\sum_{(i_1,i_2,i_3)} {\rm det}(\vv{G}(i_1,i_2,i_3)) \\
&= \binom{q}{3}^{-1} \sum_{(i_1,i_2)\neq q} {\rm det}(\vv{G}(i_1,i_2,q)) \\
& = \binom{q-1}{2}\binom{q}{3}^{-1} M_{p,1}^{(2)}(\vv{\Sigma}_q)\\
&= \binom{q-1}{2}\binom{q}{3}^{-1} \left(\frac{w^2}{p}{\rm tr}(\vv{\Lambda}_{\perp,3}^2)\right) M_{p,1}^{(2)}\\
& = \frac{3w^2}{q} \frac{1}{p}{\rm tr}(\vv{\Lambda}_{\perp,3}^2)M_{p,1}^{(2)}. 
\end{align*}
According to Theorem \ref{thm:pw}, the power function satisfies (for $p$ large enough), 
\begin{equation}\label{exm:1}
    \P_{H_1}\left(\frac{q^{1/2}p \hat{M}_p^{(3)}}{\hat{\sigma}_{p}^{(2)}}>z_{\alpha}\right) = \Phi\left(\frac{p}{\sqrt{q}}\frac{w^2}{2\beta_p^{1/2}} \frac{1}{p}{\rm tr}(\vv{\Lambda}_{\perp,3}^2) - z_{1-\alpha}\right) +o(1),
\end{equation}
where
\[
\beta_p = \frac{1}{q-1}\sum_{j=1}^{q-1} c_{j,p}^2 \frac{1}{p}{\rm tr}(\vv{\Lambda}_{I_j}^2)^2 + o(1).
\]

\end{example}

\begin{example} [$d_0=3$] \label{e2}
Suppose that for $j = 1,\ldots,q-1$, $\vv{\Sigma}_j = (\sigma_{rs,j})_{1\leq r,s\leq p}$ is a tridiagonal matrix, i.e.,
\[
\sigma_{rs,j} = \begin{cases}
1+a_j^2 + b_j^2, & \hbox{if}~r=s;\\
a_j(1+b_j), & \hbox{if}~|r-s| =1; \\
b_j, &\hbox{if}~|r-s|=2;\\
0, & \hbox{otherwise},
\end{cases}
\]
where $\{a_j\}$ and $\{b_j\}$ are two sequences of real numbers. For $j =q$, we assume that $\vv{\Sigma}_q =\vv{\Sigma}_q^w = (\sigma_{rs,q}^w)_{1\leq r,s\leq q}$ satisfies
\[
\sigma_{rs,q}^w = \begin{cases}
1+a_q^2+b_q^2 + w^2 A_0^2, & \hbox{if}~r=s;\\
a_q+ a_qb_q + wb_qA_0, & \hbox{if}~|r-s| =1; \\
b_q + wa_qA_0, &\hbox{if}~|r-s| =2;\\
wA_0, & \hbox{if}~|r-s|=3;\\
0, & \hbox{otherwise},
\end{cases}
\]
where $a_q$, $b_q$ and $A_0$ are real numbers and $w\geq 0$ is a varying non-negative coefficient.

Let $\vv{\Lambda}_0 = \vv{I}_p$ and $\vv{\Lambda}_j = (I_{|r-s|=j})_{1\leq r,s\leq p}$ for $j=1,2,3$, where $I_A$ is the indicator function of event $A$. Then, for $j=1,\ldots,q-1$, it holds that
\[
\vv{\Sigma}_j = (1+a_j^2+ b_j^2) \vv{\Lambda}_0 + a_j(1+b_j)\vv{\Lambda}_1+b_j\vv{\Lambda}_2,
\]
and for $j=q$,
\[
\vv{\Sigma}_q = (1+a_q^2+b_q^2 +w^2 A_0^2) \vv{\Lambda}_0 + (a_q+a_qb_q+wb_qA_0)\vv{\Lambda}_1 + (b_q+wa_qA_0)\vv{\Lambda}_2 +wA_0 \vv{\Lambda}_3.
\]
Hence, ${\rm dim}(\mathcal{H}_1) = 3$. Further, when $w =0$, $\{\vv{\Sigma}_j\}$ are linear combinations of $\vv{\Lambda}_0$, $\vv{\lambda}_1$  and $\vv{\Lambda}_2$ so that $H_0$ holds. Otherwise, when $w >0$, $H_1$ holds since $\vv{\Sigma}_q$ fall outside the space spanned by $\vv{\Lambda}_0, \vv{\Lambda}_1$ and $\vv{\Lambda}_2$.

Note that 
\[
\vv{G}(\{\vv{\Lambda}_j\}) = \begin{pmatrix}
1 & 0& 0 & 0\\
0  & 2-\frac{2}{p} & 0 & 0\\
0  & 0   & 2-\frac{4}{p}& 0\\
0  & 0   & 0 & 2-\frac{6}{p}
\end{pmatrix}
\]
is diagonal, which implies that $\{\vv{\Lambda}_j\}$ are orthogonal. Hence, by (\ref{eq:Mpk}), (\ref{eq:Me}) and (\ref{eq:detdecomp}), we have
\begin{align*}
M_p^{(4)} & = \binom{q}{4}^{-1}\sum_{(i_1,i_2,i_3,i_4)} {\rm det}(\vv{G}(i_1,i_2,i_3,i_4)) \\
& = \binom{q}{4}^{-1} \sum_{(i_1,i_2,i_3)\neq q} {\rm det}(\vv{G}(i_1,i_2,i_3,q)) \\
& = \frac{4w^2 A_0^2}{q}\left(2-\frac{6}{p}\right) M_{p,K}^{(2)}\\
& = \frac{8w^2A_0^2}{q} M_{p,1}^{(3)}. 
\end{align*}
According to Theorem \ref{thm:pw}, the power function satisfies (for large $p$), 
\begin{equation}\label{exm:2}
    \P_{H_1}\left(\frac{q^{1/2}p \hat{M}_p^{(4)}}{\hat{\sigma}_{p}^{(3)}}>z_{1-\alpha}\right) = \Phi\left(\frac{p}{\sqrt{q}}\frac{w^2 A_0^2}{\beta_p^{1/2}} - z_{1-\alpha}\right) +o(1),
\end{equation}
where
\[
\beta_p = \frac{1}{q-1}\sum_{j=1}^{q-1} c_{j,p}^2\left\{(1+a_j^2+b_j^2)^2 + 2a_j^2(1+b_j)^2+2b_j^2\right\}^2 + o(1).
\]
\end{example}

\subsection{Comparison with the test of proportionality hypothesis in \cite{MWY2}}

When $d_0=1$, the null hypothesis reduces to the proportionality hypothesis discussed in Section 2 of \cite{MWY2}. It is natural for us to make a comparison of two proportionality test procedures in \cite{MWY2} and this paper.
Recall that the quantities to characterize the null hypothesis in \cite{MWY2} and \ref{sec:main} are both constructed in the form
\[
\binom{q}{2}^{-1}\sum_{i<j} d(\vv{\Sigma}_i,\vv{\Sigma}_j)
\]
with $d(\cdot,\cdot)$ a distance to characterize the proportionality.
The main difference between the two methods is the ideas behind the choices of distance. 
In fact, the one used in  \cite{MWY2} is motivated by the fact that matrices $\vv{A}$ and $\vv{B}$ are mutually proportional if and only if the matrix difference
\begin{equation}\label{eq:prop}
\begin{split}
\vv{D}(\vv{A},\vv{B}) &= \left(\frac{1}{p}{\rm tr}(\vv{B})\right)\vv{A} - \left(\frac{1}{p}{\rm tr}(\vv{A})\right)\vv{B}\\
&={\rm det}\begin{pmatrix}
    \vv{A} & \vv{B} \\
    \frac{1}{p}{\rm tr}(\vv{A}) & \frac{1}{p}{\rm tr}(\vv{B}) 
\end{pmatrix} = \vv{O}.
\end{split}
\end{equation}
Thus, the corresponding distance is defined as the trace of squared $\vv{D}(\vv{A},\vv{B})$, that is,
\[
\begin{split}
d_{\rm prop}(\vv{A},\vv{B}) &= {\rm tr}(\vv{D}(\vv{A},\vv{B})\vv{D}(\vv{A},\vv{B})^\T) \\
& = {\rm tr}(\vv{A}\vv{A}^\T) \left(\frac{1}{p}{\rm tr}(\vv{B})\right)^2 + {\rm tr}(\vv{B}\vv{B}^\T) \left(\frac{1}{p}{\rm tr}(\vv{A})\right)^2 \\
& - 2{\rm tr}(\vv{A}\vv{B}^\T)\left(\frac{1}{p}{\rm tr}(\vv{A})\right)\left(\frac{1}{p}{\rm tr}(\vv{B})\right).
\end{split}
\]
It is easy too see that $\vv{D}(\vv{A},\vv{B}) = \vv{O}$ if and only if $d_{\rm prop}(\vv{A},\vv{B}) =0$.

The distance used in this paper is constructed based on the determinant of the Gram matrix of $\vv{A}$ and $\vv{B}$, that is, 
\[
\begin{split}
d(\vv{A},\vv{B}) &= p {\rm det}\begin{pmatrix}
    \frac{1}{p}{\rm tr}(\vv{A}\vv{A}^\T) & \frac{1}{p} {\rm tr}(\vv{A}\vv{B}^\T) \\
    \frac{1}{p} {\rm tr}(\vv{B}\vv{A}^\T) & \frac{1}{p}{\rm tr}(\vv{B}\vv{B}^\T).
\end{pmatrix} \\
&= p\left[\left(\frac{1}{p}{\rm tr}(\vv{A}\vv{A}^\T)\right)\left(\frac{1}{p}{\rm tr}(\vv{B}\vv{B}^\T)\right) - \left(\frac{1}{p}{\rm tr}(\vv{A}\vv{B}^\T)\right)^2\right].
   \end{split}
\]
When $\vv{A}$ and $\vv{B}$ are mutually proportional, the Gram matrix of $\vv{A}$ and $\vv{B}$ degenerates so that the determinant equals to zero. Otherwise, the Gram matrix will be full-ranked with a positive value of its determinant. 

Although these two distances both characterize the proportionality well, 
we prefer the one used in this paper because the idea of characterizing the dimension of the linear span of matrices by determinants of their Gram matrices is more general and can be easily extended to fit cases with arbitrary dimensions. In contrast, the method in the former chapter is too restrictive since it is hard to find a high-dimensional analog of (\ref{eq:prop}) to characterize cases with dimensions higher than two.

\section{Monte-Carlo Results}\label{sec:simu}

Our simulation study will focus on 'needle-in-a-haystack' problems considered in Examples \ref{e1} and \ref{e2} in the last section.
The design follows the settings in Theorem \ref{thm:pw} with dimension $p = 400$, group number $q=100$ and sample sizes $\{n_j\}$ randomly taken from integers between $200$ and $600$. In addition, samples are generated with the structure $\vv{x}=\vv{\Sigma}^{1/2}\vv{z}$, where two different types of noises $\vv{z} =(z_j)$ are considered: (i). standard normal distribution with $z_j \sim \mathcal{N}(0,1)$; (ii). centered Gamma distribution with $z_j\sim {\sf Gamma}(4,2)-2$.

Following the notations in Section \ref{ssec:power}, we have $\mathcal{H}_1 = {\rm span}\{\vv{\Sigma}_j: j =1,\ldots,99\}$. Further, as in Examples \ref{e1} and \ref{e2}, $\vv{\Sigma}_{100} = \vv{\Sigma}^w_{100}$ are chosen to be dependent on a varying coefficient $w$ taking values in $[0,1]$, such that when $w =0$, $\vv{\Sigma}_{100}$ falls in $\mathcal{H}_1$ and $H_0$ holds ; when $w>0$, $\vv{\Sigma}_{100}$ gradually deviates from $\mathcal{H}_1$ as $w$ increases to $1$. 

Other parameters are chosen as follows.
\begin{itemize}
    \item[(a)] ($d_0=2$). As in Example \ref{e1}, 
    we randomly generate three orthogonal matrices $\{\vv{U}_1,\vv{U}_2,\vv{U}_3\}$ and three diagonal matrices $\{\vv{D}_1,\vv{D}_2,\vv{D}_3\}$ with diagonal entries uniformly chosen from $(0,1)$. Then, we define $\vv{\Lambda}_j =\vv{U}_j^\T \vv{\Lambda}_j\vv{U}_j$. Meanwhile, we generate a uniform sample $\{I_1,\ldots,I_{q}\}$ from  $\{1,2\}$.
   
    \item[(b)] ($d_0=3$). As in Example \ref{e2}, we  choose $\{a_j\}$ and $\{b_j\}$ randomly from the interval $[-2,2]$, and take $A_0 = 2.5$. 
\end{itemize}

For each $w$ from $0$ to $1$ with step size $0.1$, we repeat simulations $800$ times for cases (a) and (b), each with two different noises, and derive the empirical sizes and powers. The empirical sizes under $H_0$ (or equivalently, $w=0$) are collected in Table \ref{tab:4.1}. We also display in Figure \ref{fig:4.1} the graphs of curves of empirical power functions for different types of noises and theoretical power curves derived using (\ref{exm:1}) and (\ref{exm:2}). It can be observed that our many-sample dimensionality test controls the size well under $H_0$, and effectively rejects the null when the unique outlier covariance matrix deviates from $\mathcal{H}_1$ as $w$ increases. 
The sightly conservative sizes 
are explained by the difficulty of the testing problem where in total $q=100$ populations matrices are involved for small subspace dimensions $d_0 = 2$ and $d_0=3$, respectively, whole the sample size are not very large. 

%--------------------------------------
\begin{table}
\caption{Empirical sizes of the many-sample dimensionality test.} 
\label{tab:4.1}
\centering
\begin{tabular}{@{}cccc@{}}
\hline
\multicolumn{2}{c}{Case (a)}&
\multicolumn{2}{c}{Case (b)}
 \\ [3pt]
\hline
Normal & Gamma & 
Normal & Gamma
 \\
\hline
$0.035$ & $0.0425$ & 
$0.035$ & $0.0338$
 \\
\hline
\end{tabular}
\end{table}

%----------------------------------------

\begin{figure}
\centering
\includegraphics[scale=0.37]{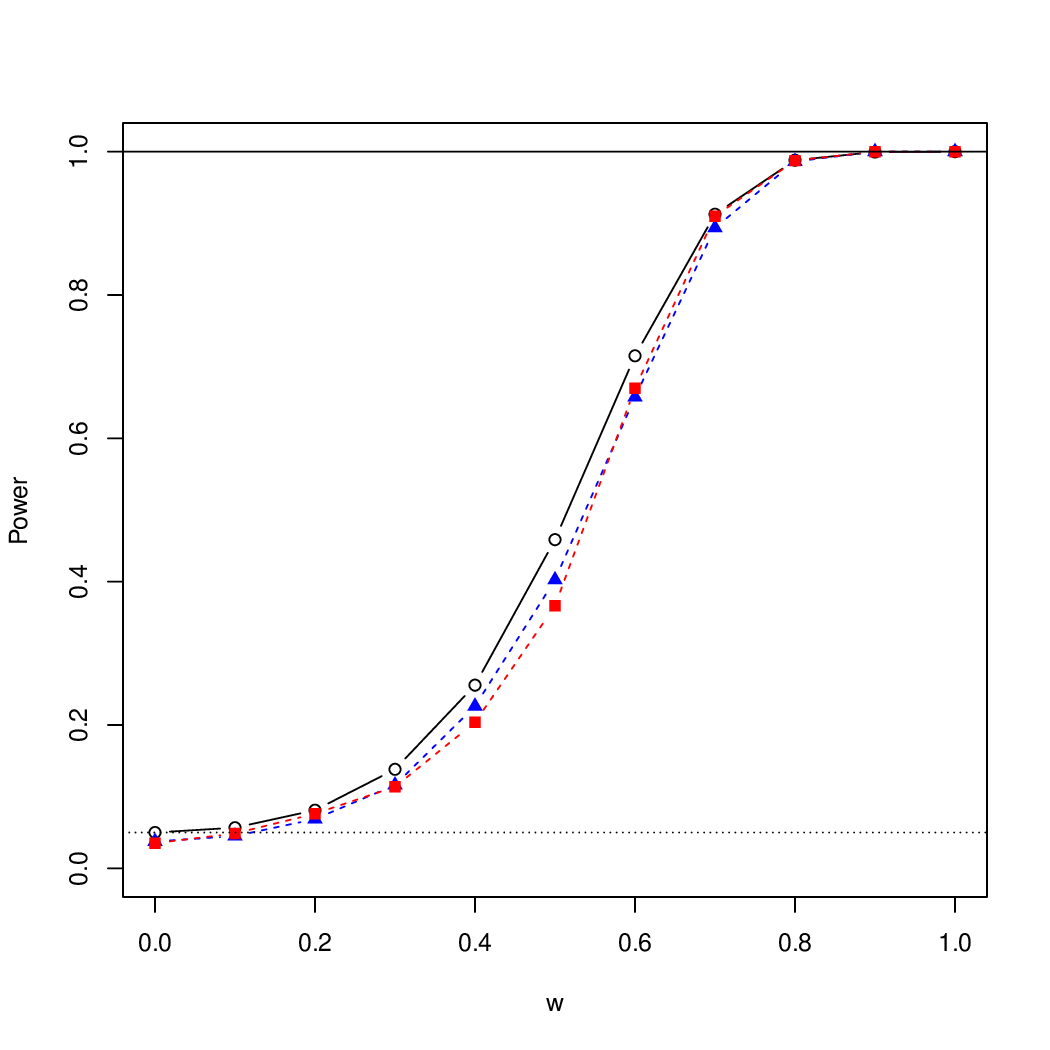}
\includegraphics[scale=0.37]{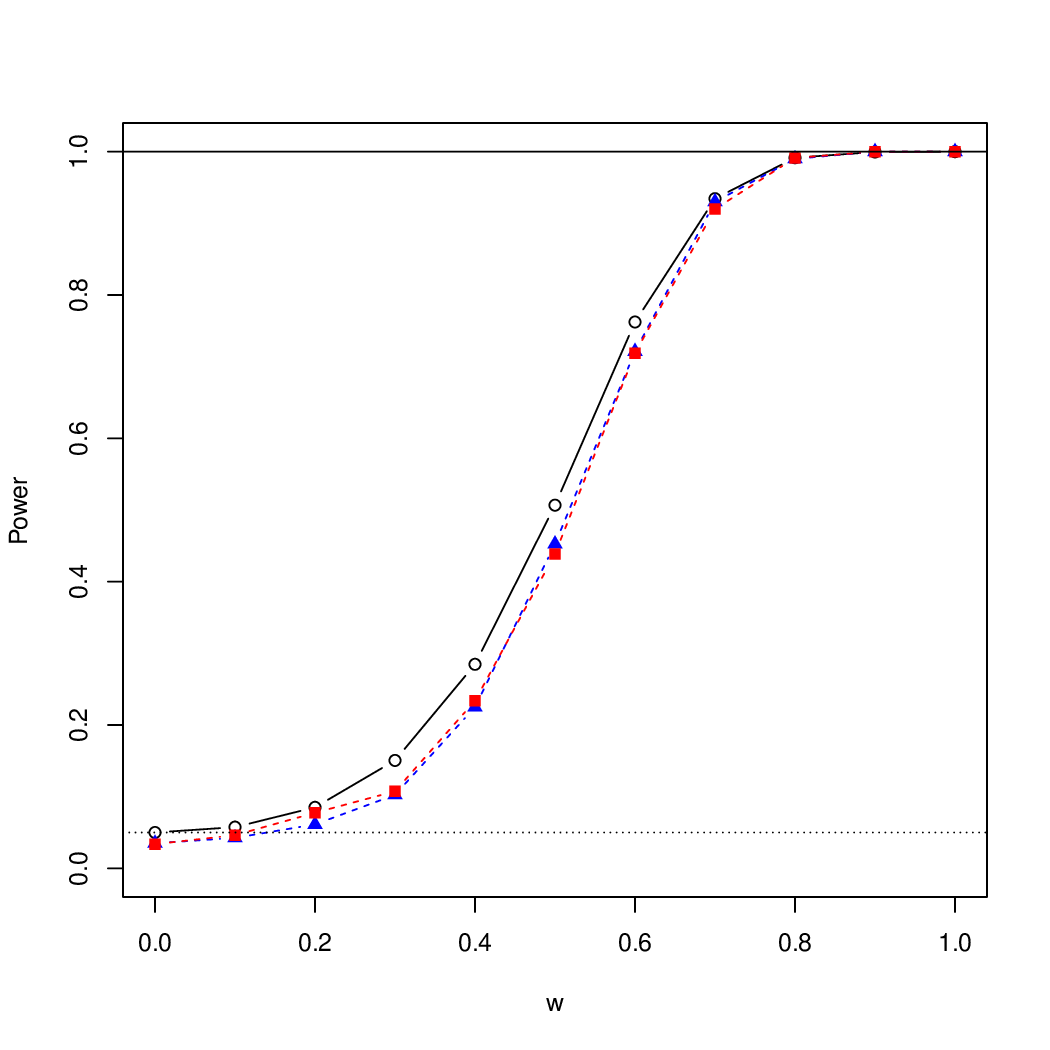}
\caption{Empirical power plot of the many-sample dimensionality test. Case (a) (left) and (b) (right) with two different noises: normal noise (red square) and Gamma noise (blue triangle). Curves with black circles stand for theoretical power functions in (\ref{exm:1}) and (\ref{exm:2}), respectively.}
\label{fig:4.1}
\end{figure}

\section{Real Data Application}\label{sec:data}

We consider a special matrix-valued gene dataset of Mouse Aging Project collected in \cite{mice07}, which measures gene expression levels from $p=46$ genes in $q=9$ distinct tissues for $40$ mice ($n=40$). 
Although $p,q$ and $n$ are not large here, the fact that three numbers are comparable so that the high-dimensional many-sample test considered in this paper is relevant.
In \cite{Touloumis20211309}, the authors  assumed that $\vv{X}$ has a Kronecker product dependence structure 
\begin{equation}\label{eq:kron}
    \vv{X} = \vv{\Sigma}_R^{1/2} \vv{Z} \vv{\Sigma}_C^{1/2},
\end{equation}
where $\vv{\Sigma}_R$ and $\vv{\Sigma}_C$ are semi-definite matrices with respective dimension $p$ and $q$, and $\vv{Z}$ is a $p\times q$ random matrices having i.i.d. entries with zero mean and unit variance. 
Under this assumption, they considered the problem of testing the diagonality hypothesis: ``$\vv{\Sigma}_C$ is diagonal'', and their procedure on the mouse aging data led to the conclusion that $\vv{\Sigma}_C$ is diagonal. 

Notice that, under the assumption (\ref{eq:kron}), the diagonality of $\vv{\Sigma}_C$ is equivalent to column independence of $\vv{X}$. Due to this reason, the proposed test in \cite{Touloumis20211309} essentially tests the hypothesis that the $q=9$ columns from the tissues are independent.  
However, the model assumption (\ref{eq:kron}) might not be reasonable. In fact, if both the assumption (\ref{eq:kron}) and  test conclusion of ``$\vv{\Sigma}_C$ is diagonal'' hold, the covariance matrices of these $q=9$ columns should be mutually proportional. Then, any proportionality test procedure should accept the null hypothesis. However, as shown in \cite{MWY2}, the proportionality test there shows a strong rejection of the null hypothesis, which suggests that the assumption (\ref{eq:kron}) might not be reasonable for this dataset.

To further investigate the dependence relationship of genes among tissues, we now assume that $\vv{X}$ satisfies the following {\em independent component model} 
    \begin{equation}\label{eq:icm}
    {\rm vec}(\vv{X}) = \vv{\Sigma}^{1/2}{\rm vec}(\vv{Z}),     
    \end{equation}
where $\vv{\Sigma}$ is $pq \times pq$ non-negative definite matrix and $\vv{Z}$ is $p\times q$ random matrix having i.i.d. entries with zero mean and unit variance. In addition, to ensure the column independence, we further assume $\vv{\Sigma}$ is block-diagonal, that is,
    \begin{equation}\label{eq:block}
    \vv{\Sigma} = {\rm diag}\{\vv{\Sigma}_{1},\ldots,\vv{\Sigma}_q\},
    \end{equation}
where $\vv{\Sigma}_j = {\rm Cov}(\vv{x}_j)$ is the covariance matrix of the $j$th column $\vv{x}_j$, $j=1,\ldots,q$. 
Note that, as a special case of  model (\ref{eq:icm}), (\ref{eq:kron}) with a diagonal $\vv{\Sigma}_C$ corresponds to the case where  ${\rm span}\{\vv{\Sigma}_j\}$ has dimension $d=1$. The rejection of the proportionality hypothesis in \cite{MWY2} suggests that the dimension of ${\rm span}(\{\vv{\Sigma}_j\})$ must be larger than one. Thus, it is natural for us
 to detect the ``true'' dimension $d$ of ${\rm span}\{\vv{\Sigma}_j\}$. 
 Observe that when $d=d_0$, it holds that
\begin{equation}\label{eq:sumkron}
    \vv{\Sigma} = \sum_{l=1}^{d_0} \vv{A}_{l} \otimes \vv{B}_l,
\end{equation}
where $\{\vv{A}_l\}_{l=1}^{d_0}$ are diagonal $q\times q$ matrices, 
and $\{\vv{B}_l\}_{l=1}^{d_0}$ are linearly independent $p\times p$ matrices. 

Note that the procedure in \cite{Touloumis20211309} essentially tests column independence of $q=9$ tissues, so their conclusion is indeed to accept the independence of the columns. Hence, we accept the column independence assumption and do our dimensionality test sequentially. Specifically, starting from $d=1$, we test the hypothesis
\[
H_0^{(d)}:~~~{\rm dim}({\rm span}(\{\vv{\Sigma}_j\})) = d~~~\hbox{versus}~~~H_1^{(d)}:~~~{\rm dim}({\rm span}(\{\vv{\Sigma}_j\})) \geq d+1
\]
successively, and stop at the first $d_0$ for which $H_0^{(d_0)}$ is accepted. Then, we assert that $d_0$ is the dimension of ${\rm span}\{\vv{\Sigma}_j\}$. Following this principle, we conduct the sequential tests and derive the corresponding $p$-values for different values of $d$ in Table \ref{tab:4.2}.
\begin{table}
\caption{The $p$-values of sequential tests for $H_0^{(d)}$ with $d = 1,2,3$.} 
\label{tab:4.2}
\centering
\begin{tabular}{@{}cccc@{}}
\hline
$d$ & 1 & 
2 & 3
 \\
\hline
$p$-value & $3.03\times 10^{-8}$ & 0.0317 & 0.368
 \\
\hline
\end{tabular}
\end{table}
It can be observed that our sequential tests strongly reject ``$d=1$'', reject ``$d=2$'', and finally accept ``$d=3$''. Hence, we conclude that the true dimension $d_0 =3$. 

To further support our findings, we treat the problem from an estimation viewpoint. We will approximate $\vv{\Sigma}$ by Kronecker products and show approximation by a sum of three Kronecker products is better than by one Kronecker product.  

The problem of approximation by Kroncker products has been extensively studied in \cite{tsiligkaridis2013covariance,van1993approximation}, in which a reshaping operator plays a prominent role. Let $\vv{M}$ be a $(pq)\times (pq)$ matrix partitioned into $q\times q$ matrices $\{\vv{M}_{ij}\}_{i,j=1}^q$ of size $p$, that is,
\[
\vv{M} = \left(\begin{matrix}
    \vv{M}_{11} & \cdots & \vv{M}_{1q} \\
    \vdots & \ddots & \vdots\\
    \vv{M}_{q1} & \cdots & \vv{M}_{qq}
\end{matrix}\right).
\]
Define the reshaping operator $\mathcal{R}: \R^{pq\times pq} \to \R^{q^2 \times p^2}$ by setting the $(i-1)q +j$ row of $\mathcal{R}(\vv{M})$ to be ${\rm vec}(\vv{M}_{ij})^\T$. It is easy to see that for a Kronecker product $\vv{C} = \vv{A}\otimes\vv{B}$, where $\vv{A}$ is a $p\times p$ matrix and $B$ is a $q\times q$ matrix, $\mathcal{R}(\vv{C}) = {\rm vec}(\vv{A}){\rm vec}(\vv{B})^\T$.
Therefore, if $\vv{\Sigma}$ satisfies (\ref{eq:sumkron}), we have
    \[
    \mathcal{R}(\vv{\Sigma}) = \sum_{l=1}^{d_0} {\rm vec}(\vv{A}_l) {\rm vec}(\vv{B}_l)^\T.
    \]
Moreover, the Frobenious norm is invariant under the operator $\mathcal{R}$ so that
    \[
    \|\vv{\Sigma} - \sum_{l=1}^{d_0} \vv{A}_l \otimes \vv{B}_l\|_F = \|\mathcal{R}(\vv{\Sigma}) - \sum_{l=1}^{d_0} {\rm vec}(\vv{A}_l) {\rm vec}(\vv{B}_l)^\T\|_F
    \]
Let $\sigma_j$ be the $j$th largest singular value of $\mathcal{R}(\vv{\Sigma})$ with $\vv{\mu}_j\in\R^{q^2}$ and $\vv{\nu}_j \in \R^{p^2}$ the corresponding left and right singular vectors. Then, the ``best'' rank-$d$ apprixmation of $\mathcal{R}(\vv{\Sigma})$ is
  \[
  \mathcal{R}^{(d)}(\vv{\Sigma}) = \sum_{l=1}^{d} \sigma_l \vv{\mu}_j \vv{\nu}_j^\T.
  \]
Since $\mathcal{R}$ is invertible, $\mathcal{R}^{-1}(\mathcal{R}^{(d)}(\vv{\Sigma}))$, the pre-image of the rank-$d$ approximation,  will be the best approximation of $\vv{\Sigma}$ in terms of a sum of $d$ Kronecker products. 

Hence, we design the following experiment:
   \begin{itemize}
       \item[(a)] Randomly split $\vv{X}_1,\ldots,\vv{X}_{40}$ into two groups with the same size $20$;

       \item[(b)] The first group is a training data set to generate rank-$1$ and rank-$3$ approximations.  In detail, for $\{\vv{X}_{1,1},\ldots,\vv{X}_{1,20}\}$, we consider its sample covariance matrix
       \[
       \hat{\vv{S}}_{1} = \frac{1}{20}\sum_{t=1}^{20} {\rm vec}(\vv{X}_{1t}) {\rm vec}(\vv{X}_{1t})^\T - \bar{\vv{X}_1}\bar{\vv{X}_1}^\T,
       \]
       where $\bar{\vv{X}_1}={\rm vec}(\frac{1}{20} \sum_{t=1}^{20} \vv{X}_{1t})$. And denote $\hat{\vv{R}}_1 = \mathcal{R}(\hat{\vv{S}}_1)$. Then, we apply the singular value decomposition to $\hat{\vv{R}}_1$ to obtain the ''best'' rank-$d$ approximation $\mathcal{R}^{(d)}$ for $d=1$ and $3$ respectively; 

       \item[(c)] The second group is a test data set to evaluate the quantity of the rank-$1$ and rank-$3$ approximations constructed in the previous step. In detail, and similarly as above, we create $\hat{\vv{R}}_2 = \mathcal{R}(\hat{\vv{S}}_2)$.

       \item[(d)] Finally, we compute $RSS_d = \|\mathcal{R}^{(d)} - \hat{\vv{R}}_2\|_F$, the Frobinous norm between $\hat{\vv{R}}_2$ and the best rank-$d$ approximation $\mathcal{R}^{(d)}$, for $d=1,3$.
   \end{itemize}

Following this procedure, we repeat different random splits 1000 times and display the scatter plot and the histogram of values of the difference $RSS_3-RSS_1$ in Figure \ref{fig:4.2}. The results show that $672$ times among 1000, $RSS_3$ is smaller than $RSS_1$, which is reflected in Figure \ref{fig:4.2} that the majority of points fall below zero. In addition, the mean and the standard deviation of $RSS_3-RSS_1$ are $-0.3713$ and $3.305$ respectively. The results support our findings that the rank-$3$ approximation is indeed better than the rank-$1$ one.

\begin{figure}
\centering
\includegraphics[width = 60 mm, height = 60 mm]{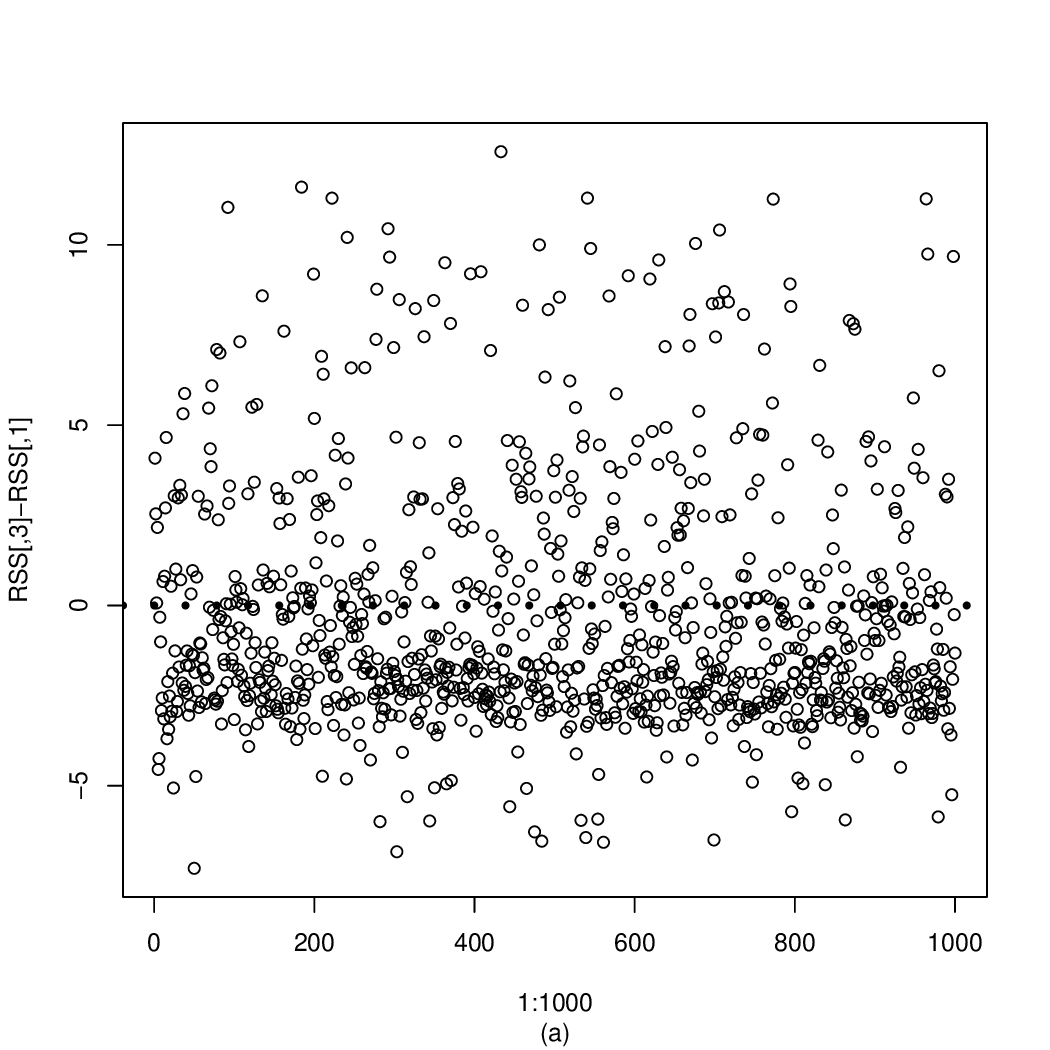}
\includegraphics[width = 60 mm, height = 60 mm]{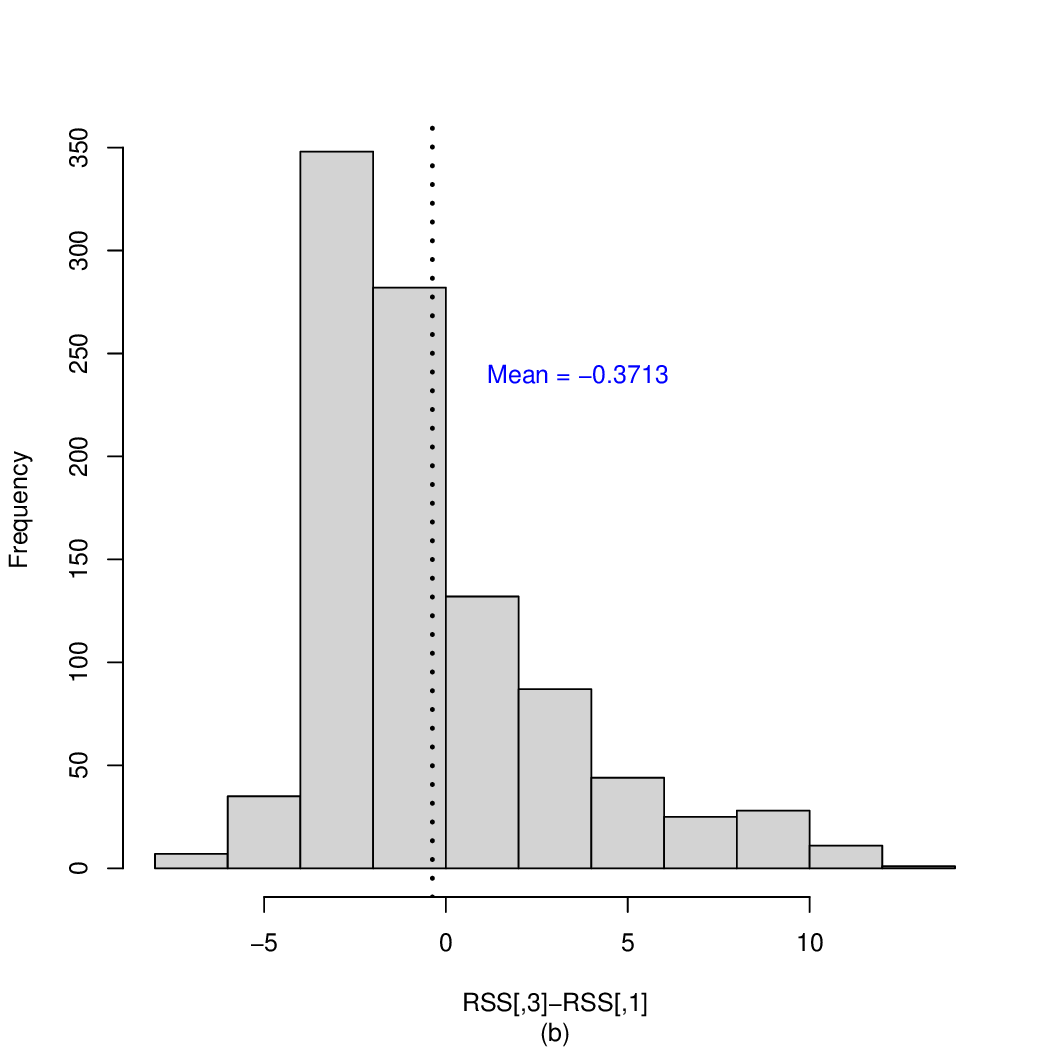}
\caption{The Scatter plot (a) and the histogram (b) of values of the difference $RSS_3-RSS_1$ for 1000 experiments. }
\label{fig:4.2}
\end{figure}

\if UNTRUE
\begin{acks}[Acknowledgments]
The authors wish to express their appreciation to the anonymous reviewers, the Associate Editor, and the Editor for their insightful feedback and constructive suggestions, which have greatly contributed to the improvement of this paper.
\end{acks}
\fi

\begin{funding}
Chen Wang was partially supported by Hong Kong RGC
General Research Fund 17301021 and National Natural Science Foundation of China Grant
72033002. Jianfeng Yao was partially supported by NSFC RFIS Grant No. 12350710179.
\end{funding}

%%%%%%%%%%%%%%%%%%%%%%%%%%%%%%%%%%%%%%%%%%%%%%
%% Supplementary Material, if any, should   %%
%% be provided in {supplement} environment  %%
%% with title and short description.        %%
%%%%%%%%%%%%%%%%%%%%%%%%%%%%%%%%%%%%%%%%%%%%%%

\bibliographystyle{imsart-nameyear} 
\bibliography{manyref,hypothesis}

@ARTICLE{Klein2022,
author={Klein, D. and Pielaszkiewicz, J. and Filipiak, K.},
title={Approximate normality in testing an exchangeable covariance structure under large- and high-dimensional settings},
journal={Journal of Multivariate Analysis},
year={2022},
volume={192},
doi={10.1016/j.jmva.2022.105049},
art_number={105049},
url={https://www.scopus.com/inward/record.uri?eid=2-s2.0-85132345998&doi=10.1016%2fj.jmva.2022.105049&partnerID=40&md5=d63ecd24de08afab7d210053168a5c86},
document_type={Article},
source={Scopus},
}

@ARTICLE{Fan2022,
author={Fan, X. and Lan, W. and Zou, T. and Tsai, C.-L.},
title={Covariance Model with General Linear Structure and Divergent Parameters},
journal={Journal of Business and Economic Statistics},
year={2022},
doi={10.1080/07350015.2022.2142593},
url={https://www.scopus.com/inward/record.uri?eid=2-s2.0-85142884638&doi=10.1080%2f07350015.2022.2142593&partnerID=40&md5=97dc6aa53469233ea998d2c1aafa916e},
document_type={Article},
source={Scopus},
}

@ARTICLE{Yinqiu2021154,
author={Yinqiu, H.E. and Gongjun, X.U. and Chong, W.U. and Pan, A.W.E.I.},
title={Asymptotically independent u-statistics in high-dimensional testing},
journal={Annals of Statistics},
year={2021},
volume={49},
number={1},
pages={154-181},
doi={10.1214/20-AOS1951},
url={https://www.scopus.com/inward/record.uri?eid=2-s2.0-85101026234&doi=10.1214%2f20-AOS1951&partnerID=40&md5=ee09ee7a62220699db2ba282f95aab23},
document_type={Article},
source={Scopus},
}

@ARTICLE{Bodnar20192977,
author={Bodnar, T. and Dette, H. and Parolya, N.},
title={Testing for independence of large dimensional vectors},
journal={Annals of Statistics},
year={2019},
volume={47},
number={5},
pages={2977-3008},
doi={10.1214/18-AOS1771},
url={https://www.scopus.com/inward/record.uri?eid=2-s2.0-85066104870&doi=10.1214%2f18-AOS1771&partnerID=40&md5=221cf21f9e0fc3d35ae885ec6950e276},
document_type={Article},
source={Scopus},
}

@ARTICLE{Xu20173208,
author={Xu, K.},
title={Testing diagonality of high-dimensional covariance matrix under non-normality},
journal={Journal of Statistical Computation and Simulation},
year={2017},
volume={87},
number={16},
pages={3208-3224},
doi={10.1080/00949655.2017.1362405},
url={https://www.scopus.com/inward/record.uri?eid=2-s2.0-85028544841&doi=10.1080%2f00949655.2017.1362405&partnerID=40&md5=129561e139a38051dd07eafb4e112671},
document_type={Article},
source={Scopus},
}

@ARTICLE{Zhong20171185,
author={Zhong, P.-S. and Lan, W. and Song, P.X.K. and Tsai, C.-L.},
title={Tests for covariance structures with high-dimensional repeated measurements},
journal={Annals of Statistics},
year={2017},
volume={45},
number={3},
pages={1185-1213},
doi={10.1214/16-AOS1481},
url={https://www.scopus.com/inward/record.uri?eid=2-s2.0-85020650658&doi=10.1214%2f16-AOS1481&partnerID=40&md5=edb847c8526130c47427c6c73cfb9a6d},
document_type={Article},
source={Scopus},
}

@ARTICLE{Yamada2017305,
author={Yamada, Y. and Hyodo, M. and Nishiyama, T.},
title={Testing block-diagonal covariance structure for high-dimensional data under non-normality},
journal={Journal of Multivariate Analysis},
year={2017},
volume={155},
pages={305-316},
doi={10.1016/j.jmva.2016.12.009},
url={https://www.scopus.com/inward/record.uri?eid=2-s2.0-85010447624&doi=10.1016%2fj.jmva.2016.12.009&partnerID=40&md5=5be9eb0fa16722b74df2b19cb9aea1b5},
document_type={Article},
source={Scopus},
}

@ARTICLE{Butucea2016164,
author={Butucea, C. and Zgheib, R.},
title={Sharp minimax tests for large Toeplitz covariance matrices with repeated observations},
journal={Journal of Multivariate Analysis},
year={2016},
volume={146},
pages={164-176},
doi={10.1016/j.jmva.2015.09.003},
url={https://www.scopus.com/inward/record.uri?eid=2-s2.0-84951805901&doi=10.1016%2fj.jmva.2015.09.003&partnerID=40&md5=a66ac7661c4709619f26eb02f81a497c},
document_type={Article},
source={Scopus},
}

@ARTICLE{Hyodo2015460,
author={Hyodo, M. and Shutoh, N. and Nishiyama, T. and Pavlenko, T.},
title={Testing block-diagonal covariance structure for high-dimensional data},
journal={Statistica Neerlandica},
year={2015},
volume={69},
number={4},
pages={460-482},
doi={10.1111/stan.12068},
url={https://www.scopus.com/inward/record.uri?eid=2-s2.0-84944157510&doi=10.1111%2fstan.12068&partnerID=40&md5=660fd3b36b799c944855f3b315f54527},
document_type={Article},
source={Scopus},
}

@ARTICLE{Lan201576,
author={Lan, W. and Luo, R. and Tsai, C.-L. and Wang, H. and Yang, Y.},
title={Testing the Diagonality of a Large Covariance Matrix in a Regression Setting},
journal={Journal of Business and Economic Statistics},
year={2015},
volume={33},
number={1},
pages={76-86},
doi={10.1080/07350015.2014.923317},
url={https://www.scopus.com/inward/record.uri?eid=2-s2.0-84921927090&doi=10.1080%2f07350015.2014.923317&partnerID=40&md5=a186eb8d0996c2a1042d5c85f14eb722},
document_type={Article},
source={Scopus},
}

@article{Hu20162281,
	author = {Hu, J. and Bai, Z.D.},
	document_type = {Article},
	doi = {10.1007/s11425-016-0131-0},
	journal = {Science China Mathematics},
	number = {12},
	pages = {2281-2300},
	source = {Scopus},
	title = {A review of 20 years of naive tests of significance for high-dimensional mean vectors and covariance matrices},
	url = {https://www.scopus.com/inward/record.uri?eid=2-s2.0-84994351123&doi=10.1007%2fs11425-016-0131-0&partnerID=40&md5=7df2f1c9fb8d86fd0d7df2f527584eda},
	volume = {59},
	year = {2016}}

@article{Cai2017423,
	author = {Cai, T. Tony},
	doi = {10.1146/annurev-statistics-060116-053754},
	journal = {Annual Review of Statistics and Its Application},
	note = {Cited by: 18},
	pages = {423 -- 446},
	publication_stage = {Final},
	source = {Scopus},
	title = {Global testing and large-scale multiple testing for high-dimensional covariance structures},
	type = {Review},
	url = {https://www.scopus.com/inward/record.uri?eid=2-s2.0-85015154229&doi=10.1146%2fannurev-statistics-060116-053754&partnerID=40&md5=626a607be24af05c63eba67c8e936949},
	volume = {4},
	year = {2017}}

@article{Touloumis20211309,
	author = {Touloumis, Anestis and Marioni, John C. and Tavar{\'e}, Simon},
	doi = {10.5705/ss.202018.0268},
	journal = {Statistica Sinica},
	note = {Cited by: 1; All Open Access, Green Open Access},
	number = {3},
	pages = {1309 -- 1329},
	publication_stage = {Final},
	source = {Scopus},
	title = {Hypothesis testing for the covariance matrix in high-dimensional transposable data with kronecker product dependence structure},
	type = {Article},
	url = {https://www.scopus.com/inward/record.uri?eid=2-s2.0-85114178296&doi=10.5705%2fss.202018.0268&partnerID=40&md5=90ee70c9783fb704359baa92eafd4de2},
	volume = {31},
	year = {2021}}

@article{Bai2021701,
	author = {Bai, Zhidong and Hu, Jiang and Wang, Chen and Zhang, Chao},
	doi = {10.1007/s00362-019-01110-1},
	journal = {Statistical Papers},
	note = {Cited by: 1},
	number = {2},
	pages = {701 -- 719},
	publication_stage = {Final},
	source = {Scopus},
	title = {Test on the linear combinations of covariance matrices in high-dimensional data},
	type = {Article},
	url = {https://www.scopus.com/inward/record.uri?eid=2-s2.0-85080984301&doi=10.1007%2fs00362-019-01110-1&partnerID=40&md5=2eceee105039c4186060579b73fba303},
	volume = {62},
	year = {2021}}

@article{Ahmad2017500,
	author = {Ahmad, M. Rauf},
	doi = {10.1111/sjos.12262},
	journal = {Scandinavian Journal of Statistics},
	note = {Cited by: 11},
	number = {2},
	pages = {500 -- 523},
	publication_stage = {Final},
	source = {Scopus},
	title = {Location-invariant Multi-sample U-tests for Covariance Matrices with Large Dimension},
	type = {Article},
	url = {https://www.scopus.com/inward/record.uri?eid=2-s2.0-85019083383&doi=10.1111%2fsjos.12262&partnerID=40&md5=75da5dabf49dff1aa2234aada8787393},
	volume = {44},
	year = {2017}}

@article{Bai20093822,
	author = {Bai, Zhidong and Jiang, Dandan and Yao, Jian-Feng and Zheng, Shurong},
	doi = {10.1214/09-AOS694},
	journal = {Annals of Statistics},
	note = {Cited by: 145; All Open Access, Bronze Open Access, Green Open Access},
	number = {6 B},
	pages = {3822 -- 3840},
	publication_stage = {Final},
	source = {Scopus},
	title = {Corrections to LRT on large-dimensional covariance matrix by RMT},
	type = {Article},
	url = {https://www.scopus.com/inward/record.uri?eid=2-s2.0-73949135723&doi=10.1214%2f09-AOS694&partnerID=40&md5=d590f87ef6fb0f6c35bc3c52da2312da},
	volume = {37},
	year = {2009}}

@article{Li20162973,
	author = {Li, Zeng and Yao, Jianfeng},
	doi = {10.1214/16-EJS1199},
	journal = {Electronic Journal of Statistics},
	note = {Cited by: 6; All Open Access, Gold Open Access, Green Open Access},
	number = {2},
	pages = {2973 -- 3010},
	publication_stage = {Final},
	source = {Scopus},
	title = {Testing the sphericity of a covariance matrix when the dimension is much larger than the sample size},
	type = {Article},
	url = {https://www.scopus.com/inward/record.uri?eid=2-s2.0-84994614117&doi=10.1214%2f16-EJS1199&partnerID=40&md5=a868db7b71f369b960cccb8eddaac96e},
	volume = {10},
	year = {2016}}

@article{Ledoit20021081,
	author = {Ledoit, Olivier and Wolf, Michael},
	doi = {10.1214/aos/1031689018},
	journal = {Annals of Statistics},
	note = {Cited by: 244; All Open Access, Bronze Open Access, Green Open Access},
	number = {4},
	pages = {1081 -- 1102},
	publication_stage = {Final},
	source = {Scopus},
	title = {Some hypothesis tests for the covariance matrix when the dimension is large compared to the sample size},
	type = {Review},
	url = {https://www.scopus.com/inward/record.uri?eid=2-s2.0-0036392431&doi=10.1214%2faos%2f1031689018&partnerID=40&md5=39d441c68d3c5ad0780f76cb88e676ec},
	volume = {30},
	year = {2002}}

@article{Coelho2010711,
	author = {Coelho, Carlos A. and Arnold, Barry C. and Marques, Filipe J.},
	doi = {10.1080/15598608.2010.10412014},
	journal = {Journal of Statistical Theory and Practice},
	note = {Cited by: 16},
	number = {4},
	pages = {711 -- 725},
	publication_stage = {Final},
	source = {Scopus},
	title = {Near-exact distributions for certain likelihood ratio test statistics},
	type = {Article},
	url = {https://www.scopus.com/inward/record.uri?eid=2-s2.0-85008810800&doi=10.1080%2f15598608.2010.10412014&partnerID=40&md5=f723d0830b10126cddb4a46312beffb8},
	volume = {4},
	year = {2010}}

@article{Coelho2012627,
	author = {Coelho, Carlos A. and Marques, Filipe J.},
	doi = {10.1007/s00180-011-0281-1},
	journal = {Computational Statistics},
	note = {Cited by: 24},
	number = {4},
	pages = {627 -- 659},
	publication_stage = {Final},
	source = {Scopus},
	title = {Near-exact distributions for the likelihood ratio test statistic to test equality of several variance-covariance matrices in elliptically contoured distributions},
	type = {Article},
	url = {https://www.scopus.com/inward/record.uri?eid=2-s2.0-84868346841&doi=10.1007%2fs00180-011-0281-1&partnerID=40&md5=eba5b9f4224152b53a15c30121e56eca},
	volume = {27},
	year = {2012}}

@article{Coelho2010583,
	author = {Coelho, Carlos A. and Marques, Filipe J.},
	doi = {10.1016/j.jmva.2009.09.012},
	journal = {Journal of Multivariate Analysis},
	note = {Cited by: 17; All Open Access, Bronze Open Access},
	number = {3},
	pages = {583 -- 593},
	publication_stage = {Final},
	source = {Scopus},
	title = {Near-exact distributions for the independence and sphericity likelihood ratio test statistics},
	type = {Article},
	url = {https://www.scopus.com/inward/record.uri?eid=2-s2.0-72549084495&doi=10.1016%2fj.jmva.2009.09.012&partnerID=40&md5=2a9cbd87a403d7547b1992efc6d0e037},
	volume = {101},
	year = {2010}}

@article{Marques2008726,
	author = {Marques, Filipe J. and Coelho, Carlos A.},
	doi = {10.1016/j.jspi.2007.01.002},
	journal = {Journal of Statistical Planning and Inference},
	note = {Cited by: 18},
	number = {3},
	pages = {726 -- 741},
	publication_stage = {Final},
	source = {Scopus},
	title = {Near-exact distributions for the sphericity likelihood ratio test statistic},
	type = {Article},
	url = {https://www.scopus.com/inward/record.uri?eid=2-s2.0-36049011299&doi=10.1016%2fj.jspi.2007.01.002&partnerID=40&md5=3f565f06c116e328923d8da088a65fd3},
	volume = {138},
	year = {2008}}

@article{Ahmad20152619,
	author = {Ahmad, M. Rauf and Rosen, D. von},
	doi = {10.1080/00949655.2014.948441},
	journal = {Journal of Statistical Computation and Simulation},
	note = {Cited by: 2},
	number = {13},
	pages = {2619 -- 2631},
	publication_stage = {Final},
	source = {Scopus},
	title = {Tests for high-dimensional covariance matrices using the theory of U-statistics},
	type = {Article},
	url = {https://www.scopus.com/inward/record.uri?eid=2-s2.0-84930575733&doi=10.1080%2f00949655.2014.948441&partnerID=40&md5=a331462852f75c043cd932c61006cf7e},
	volume = {85},
	year = {2015}}

@article{Ishii201999,
	author = {Ishii, Aki and Yata, Kazuyoshi and Aoshima, Makoto},
	doi = {10.1016/j.jspi.2019.02.002},
	journal = {Journal of Statistical Planning and Inference},
	note = {Cited by: 11; All Open Access, Green Open Access},
	pages = {99 -- 111},
	publication_stage = {Final},
	source = {Scopus},
	title = {Equality tests of high-dimensional covariance matrices under the strongly spiked eigenvalue model},
	type = {Article},
	url = {https://www.scopus.com/inward/record.uri?eid=2-s2.0-85062038464&doi=10.1016%2fj.jspi.2019.02.002&partnerID=40&md5=a340f6b2fca4d362762405d9a6dcce17},
	volume = {202},
	year = {2019}}

@article{Li2012908,
	author = {Li, Jun and Chen, Song Xi},
	doi = {10.1214/12-AOS993},
	journal = {Annals of Statistics},
	note = {Cited by: 131; All Open Access, Bronze Open Access, Green Open Access},
	number = {2},
	pages = {908 -- 940},
	publication_stage = {Final},
	source = {Scopus},
	title = {Two sample tests for high-dimensional covariance matrices},
	type = {Article},
	url = {https://www.scopus.com/inward/record.uri?eid=2-s2.0-84871930256&doi=10.1214%2f12-AOS993&partnerID=40&md5=461490c07503fd5bca4ec4be3d54ca7c},
	volume = {40},
	year = {2012}}

@article{Qiu20121285,
	author = {Qiu, Yumou and Chen, Song Xi},
	doi = {10.1214/12-AOS1002},
	journal = {Annals of Statistics},
	note = {Cited by: 45; All Open Access, Bronze Open Access, Green Open Access},
	number = {3},
	pages = {1285 -- 1314},
	publication_stage = {Final},
	source = {Scopus},
	title = {Test for bandedness of high-dimensional covariance matrices and bandwidth estimatio},
	type = {Article},
	url = {https://www.scopus.com/inward/record.uri?eid=2-s2.0-84872061647&doi=10.1214%2f12-AOS1002&partnerID=40&md5=388f13b3161692e374e8f97d4a118269},
	volume = {40},
	year = {2012}}

@article{Lai2022,
	author = {Lai, J. and Wang, X. and Zhao, K. and Zheng, S.},
	document_type = {Article},
	doi = {10.1007/s11749-022-00842-x},
	journal = {Test},
	source = {Scopus},
	title = {Block-diagonal test for high-dimensional covariance matrices},
	url = {https://www.scopus.com/inward/record.uri?eid=2-s2.0-85144855542&doi=10.1007%2fs11749-022-00842-x&partnerID=40&md5=52cc97ea5875383d7d8f8eacc2ba3f98},
	year = {2022}}

@article{Cheng2020,
	art_number = {106962},
	author = {Cheng, G. and Liu, B. and Tian, G. and Zheng, S.},
	document_type = {Article},
	doi = {10.1016/j.csda.2020.106962},
	journal = {Computational Statistics and Data Analysis},
	source = {Scopus},
	title = {Testing proportionality of two high-dimensional covariance matrices},
	url = {https://www.scopus.com/inward/record.uri?eid=2-s2.0-85084176860&doi=10.1016%2fj.csda.2020.106962&partnerID=40&md5=506fc1e8290abd267373cae69d9b42b6},
	volume = {150},
	year = {2020}}

@article{Zheng20193300,
	author = {Zheng, S. and Chen, Z. and Cui, H. and Li, R.},
	document_type = {Article},
	doi = {10.1214/18-AOS1779},
	journal = {Annals of Statistics},
	number = {6},
	pages = {3300-3334},
	source = {Scopus},
	title = {Hypothesis testing on linear structures of high-dimensional covariance matrix},
	url = {https://www.scopus.com/inward/record.uri?eid=2-s2.0-85079179857&doi=10.1214%2f18-AOS1779&partnerID=40&md5=3f22fb8454656f66b4cc1543781c57e6},
	volume = {47},
	year = {2019}}

@article{Liu2013293,
	author = {Liu, B. and Xu, L. and Zheng, S. and Tian, G.-L.},
	document_type = {Article},
	doi = {10.1016/j.jmva.2014.06.008},
	journal = {Journal of Multivariate Analysis},
	pages = {293-308},
	source = {Scopus},
	title = {A new test for the proportionality of two large-dimensional covariance matrices},
	url = {https://www.scopus.com/inward/record.uri?eid=2-s2.0-84955718390&doi=10.1016%2fj.jmva.2014.06.008&partnerID=40&md5=376f6a0b111bfc8e89ace024179696f9},
	volume = {131},
	year = {2013}}

@book{vander2000,
	author = {Van der Vaart, Aad W},
	publisher = {Cambridge university press},
	title = {Asymptotic statistics},
	volume = {3},
	year = {2000}}

@article{mice07,
	author = {Zahn, J.M. and Poosala, S. and Owen, A.B. and Ingram, D.K. and Lustig, A. and Carter, A. and Weeraratna, A.T. and Taub, D.D. and Gorospe, M. and Mazan-Mamczarz, K. and Lakatta, E.G. and Boheler, K.R. and Xu, X. and Mattson, M.P. and Falco, G. and Ko, M.S.H. and Schlessinger, D. and Firman, J. and Kummerfeld, S.K. and Wood III, W.H. and Zonderman, A.B. and Kim, S.K. and Becker, K.G.},
	document_type = {Article},
	doi = {10.1371/journal.pgen.0030201},
	journal = {PLoS Genetics},
	number = {11},
	pages = {2326-2337},
	title = {AGEMAP: A gene expression database for aging in mice},
	volume = {3},
	year = {2007}}

@article{tsiligkaridis2013covariance,
  title={Covariance estimation in high dimensions via kronecker product expansions},
  author={Tsiligkaridis, Theodoros and Hero, Alfred O},
  journal={IEEE Transactions on Signal Processing},
  volume={61},
  number={21},
  pages={5347--5360},
  year={2013},
  publisher={IEEE}
}

@article{MWY2,
  title={Many-sample tests for equality and proportionality hypotheses for covariance matrices between groups},
  author={Mei, Tianxing and Wang, Chen and Yao, Jianfeng},
  year={2022}
}

@book{van1993approximation,
  title={Approximation with Kronecker products},
  author={Van Loan, Charles F and Pitsianis, Nikos},
  year={1993},
  publisher={Springer}
}

@book{muirhead2009aspects,
  title={Aspects of multivariate statistical theory},
  author={Muirhead, Robb J},
  year={2009},
  publisher={John Wiley \& Sons}
}

@book{anderson2003,
      title={An Introduction to Multivariate Statistical Analysis},
      author={Anderson, Theodore Wilbur},
      year={2003},
      publisher ={Wiley-Interscience}
      }

@article{100gene,
title = {A global reference for human genetic variation},
author = {1000 Genomes Project Consortium},
journal = {Nature},
year = {2015},
volume={526},
  number={7571},
  pages={68--74},
  doi = {10.1038/nature15393}
}

@article{Chen2010808,
	author = {Chen, S.X. and Qin, Y.-L.},
	doi = {10.1214/09-AOS716},
	journal = {Annals of Statistics},
	number = {2},
	pages = {808-835},
	title = {A two-sample test for high-dimensional data with applications to gene-set testing},
	volume = {38},
	year = {2010}}

\newpage
\appendix

\begin{supplement}
\stitle{Supplement to ``Many-sample tests for the dimensionality hypotheses for large covariance matrices among groups"}
\sdescription{In this supplement, we provide all technical proofs.}
\end{supplement}

\section{Auxiliary Lemmas}\label{app:A}

Let $\vv{x}_1,\ldots,\vv{x}_n$ be an i.i.d. sample from a $p$-dimensional population $\vv{x} =\vv{\Sigma}^{1/2} \vv{z}$,  where $\vv{\Sigma}$ is a semi-definite matrix and $\vv{z}=(z_{j})$ has i.i.d entries with zero mean, unit variance and finite fourth moment $\nu_4$.

For any semi-definite matrix $\vv{\Lambda}$, and $p\times p$ symmetric matrices $\vv{A}_1$ and $\vv{A}_2$, we denote
\begin{equation}
    \mu(\vv{A}_1,\vv{A}_2|\vv{\Lambda}) = \frac{2}{p}{\rm tr}(\vv{\Lambda}\vv{A}_1\vv{\Lambda}\vv{A}_2)+\frac{(\nu_4-3)}{p}{\rm tr}(\mathcal{D}(\vv{\Lambda}^{1/2}\vv{A}_1\vv{\Lambda}^{1/2})\mathcal{D}(\vv{\Lambda}^{1/2}\vv{A}_2\vv{\Lambda}^{1/2}))
\end{equation}
where $\mathcal{D}(\vv{A})$ is a diagonal matrix made by diagonal entries of $\vv{A}$.
It is directly to see that when $\vv{\Lambda}$ is fixed, $\mu(\cdot,\cdot|\vv{\Lambda})$ is a non-negative definite bi-linear function on the product space of symmetric matrices. 
Moreover, when $\vv{\Lambda}$, $\vv{A}_1$ and $\vv{A}_2$ are bounded uniformly in operator norm $\|\cdot\|$ with respect to $p$, so is $\mu(\vv{A}_1,\vv{A}_2|\vv{\Sigma})$.

Let $\vv{S}_n$ be the 
sample covariance matrix, that is, $\vv{S}_n=\frac{1}{n}\sum_{j=1}^n \vv{x}_i\vv{x}_i^\T$. 

In this section, we focus on moments of the following statistics related to $\vv{S}_n$:
\[
\frac{1}{p}{\rm tr}(\vv{S}_n^2)~~~\hbox{and}~~~~\left(\frac{1}{p}{\rm tr}(\vv{S}_n\vv{A}_1)\right)\left(\frac{1}{p}{\rm tr}(\vv{S}_n\vv{A}_2)\right).
\]

\begin{lem}
For any symmetric matrices $\vv{A}_1$ and $\vv{A}_2$, it holds that
\begin{equation}\label{eq:A.1}
\begin{split}
    \E\left[\frac{1}{p}{\rm tr}(\vv{S}_n^2)\right] & = \frac{1}{p}{\rm tr}(\vv{\Sigma}^2) + \frac{p}{n}\left(\frac{1}{p}{\rm tr}(\vv{\Sigma})\right)^2 \\
    & + \frac{1}{n}\left\{\frac{1}{p}{\rm tr}(\vv{\Sigma}^2)+\frac{(\nu_4-3)}{p}{\rm tr}(\mathcal{D}(\vv{\Sigma})^2)\right\}.
\end{split}
\end{equation}

\begin{equation}\label{eq:A.2}
\begin{split}
\mathbb{E}\left[\left(\frac{1}{p}{\rm tr}(\vv{S}_{n}\vv{A}_1)\right)\left(\frac{1}{p}{\rm tr}(\vv{S}_{n}\vv{A}_2)\right)\right]
&=\left(\frac{1}{p}{\rm tr}(\vv{\Sigma}\vv{A}_1)\right)\left(\frac{1}{p}{\rm tr}(\vv{\Sigma}\vv{A}_2)\right)\\
&+\frac{1}{pn}\mu(\vv{A}_1,\vv{A}_2|\vv{\Sigma}).
\end{split}
\end{equation}

\end{lem}

\begin{lem}
Let $\vv{B}_1,\ldots,\vv{B}_4$ be $p\times p$ symmetric matrices. Then,
\begin{equation*}
\begin{split}
{\rm Var}\left({\rm tr}(\vv{S}_{n}^2)\right)&=\frac{1}{n^4}\bigg\{ 4n{\rm tr}(\vv{\Sigma})^2 +8n^2 {\rm tr}(\vv{\Sigma})\mu(\vv{\Sigma},\vv{I}_p|\vv{\Sigma})\\
&~~~~~~~~~+4n^2{\rm tr}(\vv{\Sigma}^2)^2+4n^3\mu(\vv{\Sigma},\vv{\Sigma}|\vv{\Sigma})\bigg\} + o(1).
\end{split}
\end{equation*}
\begin{equation}\label{eq:A.4}
\begin{split}
&{\rm Cov}\left(\frac{1}{p}{\rm tr}(\vv{S}_{n}\vv{B}_1){\rm tr}(\vv{S}_n\vv{B}_2),\frac{1}{p}{\rm tr}(\vv{S}_{n}\vv{B}_3){\rm tr}(\vv{S}_n\vv{B}_4)\right)\\
=&\frac{1}{pn}\bigg\{{\rm tr}(\vv{\Sigma}\vv{B}_2){\rm tr}(\vv{\Sigma}\vv{B}_4)\mu(\vv{B}_1,\vv{B}_3|\vv{\Sigma})+{\rm tr}(\vv{\Sigma}\vv{B}_2){\rm tr}(\vv{\Sigma}\vv{B}_3)\mu(\vv{B}_1,\vv{B}_4|\vv{\Sigma})\\
&~~~~~+{\rm tr}(\vv{\Sigma}\vv{B}_1){\rm tr}(\vv{\Sigma}\vv{B}_4)\mu(\vv{B}_2,\vv{B}_3|\vv{\Sigma})+{\rm tr}(\vv{\Sigma}\vv{B}_1){\rm tr}(\vv{\Sigma}\vv{B}_3)\mu(\vv{B}_2,\vv{B}_4|\vv{\Sigma})\bigg\}+o(1).
\end{split}
\end{equation}
\begin{equation*}
\begin{split}
&{\rm Cov}\left({\rm tr}(\vv{S}_{n}^2),\frac{1}{p}{\rm tr}(\vv{S}_{n}\vv{B}_1){\rm tr}(\vv{S}_{n}\vv{B}_2)\right)\\
=&\frac{2}{n^4}\bigg\{n^2\cdot{\rm tr}(\vv{\Sigma}\vv{B}_2){\rm tr}(\vv{\Sigma})\mu(\vv{B}_1,\vv{I}_p|\vv{\Sigma})+n^2\cdot{\rm tr}(\vv{\Sigma}\vv{B}_1){\rm tr}(\vv{\Sigma})\mu(\vv{B}_2,\vv{I}_p|\vv{\Sigma})\\
&~~+n^3\cdot {\rm tr}(\vv{\Sigma}\vv{B}_2)\mu(\vv{B}_1,\vv{I}_p|\vv{\Sigma})+n^3\cdot {\rm tr}(\vv{\Sigma}\vv{B}_1)\mu(\vv{B}_2,\vv{I}_p|\vv{\Sigma})\bigg\}+o(1).
\end{split}
\end{equation*}
Consequently, it also holds that
\begin{equation}\label{eq:A.5}
\begin{split}
{\rm Var}\left({\rm tr}(\vv{S}_{n}^2)-c_{p}\frac{1}{p}{\rm tr}(\vv{S}_{n})^2\right)=4c_{p}^2\left(\frac{1}{p}{\rm tr}(\vv{\Sigma}^2)\right)^2+4c_{p}\mu(\vv{\Sigma},\vv{\Sigma}|\vv{\Sigma})
\end{split}
\end{equation}
\begin{equation}\label{eq:A.6}
\begin{split}
&{\rm Cov}\left({\rm tr}(\vv{S}_{n}^2)-c_{p}\frac{1}{p}{\rm tr}(\vv{S}_{n})^2,\frac{1}{p}{\rm tr}(\vv{S}_{n}\vv{B}_1){\rm tr}(\vv{S}_{n}\vv{B}_2)\right)\\
&=\frac{2}{n}\bigg\{{\rm tr}(\vv{\Sigma}\vv{B}_2)\mu(\vv{\Sigma},\vv{B}_1|\vv{\Sigma})+{\rm tr}(\vv{\Sigma}\vv{B}_1)\mu(\vv{\Sigma},\vv{B}_2|\vv{\Sigma})\bigg\}+o(1)\\
\end{split}
\end{equation}
where $c_p = p/n$.
\end{lem}

\section{Proof of Lemma \ref{lem:detG}}\label{app:B}

This subsection is devoted to prove Lemma \ref{lem:detG}. Before that, we show several important lemmas that are important for our later discussion.

\begin{lem}\label{lem:B}
Suppose that Assumptions \ref{assm:2.1}---\ref{assm:2.4} hold.
Then, for any positive integer $m$ and any distinct combinations $(j_1,\ldots,j_m)\in \{1,\ldots,q\}$, it holds that
\begin{equation}\label{eq:B.2}
\mathbb{E}\left[\hat{G}_{j_1,j_2}\cdots\hat{G}_{j_{m-1},j_m}\cdot\hat{G}_{j_m,j_1}\right]=G_{j_1,j_2}\cdots G_{l_{m-1},l_m}\cdot G_{l_m,l_1}+O(\frac{1}{p^2}),
\end{equation}
where the reminder $O(p^{-2})$ can be uniformly bounded by $C/p^2$ for certain positive constant $C$ dependent only on $m$.  
\end{lem}

\begin{proof}
For $m=1$, according to discussion in Section \ref{ssec:sgm}, we have $\mathbb{E}[\hat{G}_{j_1,j_1}]=G_{j_1,j_1}+O(\frac{1}{p^2})$ so that (\ref{eq:B.2}) holds.;

For $m\geq 2$, indices $\{i_1,\ldots,i_k\}$ must be distinct. Then, we prove (\ref{eq:B.2}) by taking conditional expectation with respect to $\{\mathcal{F}_{k}\}$ repeatedly, where $\mathcal{F}_{k}=\sigma(\vv{X}_{j_1,p},\ldots,\vv{X}_{j_{k-1},p})$. By applying (\ref{eq:A.2}), we have
\begin{align*}
&\E\left[\hat{G}_{j_1,j_2}\cdots\hat{G}_{j_{m-1},j_m}\cdot\hat{G}_{j_m,j_1}|\mathcal{F}_{m-1}\right]\\
=&\hat{G}_{j_1,j_2}\cdots\hat{G}_{j_{m-2},j_{m-1}}\left[\left(\frac{1}{p}{\rm tr}(\vv{S}_{j_{m-1},p}\vv{\Sigma}_{j_m})\right)\left(\frac{1}{p}{\rm tr}(\vv{S}_{j_{1},p}\vv{\Sigma}_{j_m})\right)+O(\frac{1}{pn_m})\right],
\end{align*}
where the reminder can be uniformly bounded by $C/p^2$ for certain positive constant $C$ independent of $\{p,q,n_j\}$. Moreover, by Assumption \ref{assm:2.2}, we have
\[
\sup_{(k,j)}\frac{1}{p}{\rm tr}(\vv{S}_{l_k,p}\vv{S}_{l_j,p})\leq\sup_{(k,j)}\|\vv{S}_{l_k,p}\|\|\vv{S}_{l_j,p}\|\leq C'.
\]
Then, the product
\[
|\hat{G}_{l_1,l_2}\cdots\hat{G}_{l_{m-1},j_{m-1}}||O(\frac{1}{pn_m})|\leq \frac{C'C}{pn_m} = O(\frac{1}{p^2})
\]
where the reminder can still be uniformly bounded. Thus, we conclude that
\begin{align*}
&\E\left[\hat{G}_{j_1,j_2}\cdots\hat{G}_{j_{m-1},j_m}\cdot\hat{G}_{j_m,j_1}|\mathcal{F}_{m-1}\right]\\
=&\hat{G}_{j_1,j_2}\cdots\hat{G}_{j_{m-2},j_{m-1}}\left(\frac{1}{p}{\rm tr}(\vv{S}_{j_{m-1},p}\vv{\Sigma}_{j_m})\right)\left(\frac{1}{p}{\rm tr}(\vv{S}_{j_{1},p}\vv{\Sigma}_{j_m}\right)+O(\frac{1}{p^2}),
\end{align*}
In other words, taking conditional expectation with respect to $\mathcal{F}_{m-1}$ of the cyclic product is asymptotically equivalent to replace $\vv{S}_{j_{m},p}$ by $\vv{\Sigma}_{j_{m},p}$ in $\hat{G}_{j_{m-1},j_m}$ and $\hat{G}_{j_m,j_1}$.

Next, we take conditional expectation with respect to $\{\mathcal{F}_k\}$ sequentially and repeat using (\ref{eq:A.2}). Consequently, the operation sequentially replaces $\vv{S}_{j_k}$'s by $\vv{\Sigma}_{j_k}$'s in the expression, and the remainder is also uniformly bounded in order $O(1/p^2)$ due to Assumptions \ref{assm:2.2}---\ref{assm:2.4}. The procedure finally leads to (\ref{eq:B.2}), which completes the proof. 
\end{proof}
\begin{corollary}\label{cor:B1}
Suppose that Assumptions \ref{assm:2.1}---\ref{assm:2.4} hold.
Let $m$ be a positive integer and $(j_1,\ldots,j_m)\in \{1,\ldots,q\}$ be a combination of distinct indices. Then, any integer $k$ less than $m$ and $(l_1,\ldots,l_k)\subset(j_1,\ldots,j_m)$, we have
\begin{equation}\label{eq:B.3}
\mathbb{E}\left[\hat{G}_{j_1,j_2}\cdots\hat{G}_{j_{m-1},j_m}\cdot\hat{G}_{j_m,j_1}|\mathcal{F}_{l_1,\ldots,l_k}\right]=\tilde{G}_{j_1,j_2}\cdots \tilde{G}_{j_{m},j_1}+O(\frac{1}{p^2}),
\end{equation}
where $\mathcal{F}_{l_1,\ldots,l_k} =\sigma(\vv{X}_{l_1,p},\ldots,\vv{X}_{l_k,p})$,  $\tilde{G}_{j_r,j_s}$'s are obtained by replacing those $\vv{S}_{j,p}$'s with $j$ not in $\{l_1,\ldots,l_k\}$ by $\vv{\Sigma}_j$'s in $\hat{G}_{j_r,j_s}$, that is,
\[
\tilde{G}_{j_s,j_t} =\begin{cases}
    \cfrac{1}{p}{\rm tr}(\vv{\Sigma}_{j_s}\vv{\Sigma}_{j_t}),&~\hbox{if neither $j_t$ nor $j_s$ are in $\{l_1,\ldots,l_k\}$};\\
    \cfrac{1}{p}{\rm tr}(\vv{S}_{j_t,p}\vv{\Sigma}_{j_s}),&~\hbox{if $j_t\in\{l_1,\ldots,l_k\}$ but $j_s$ not};\\
    \cfrac{1}{p}{\rm tr}(\vv{\Sigma}_{j_t}\vv{S}_{j_s,p}),&~\hbox{if $j_s\in\{l_1,\ldots,l_k\}$ but $j_t$ not};\\
    \cfrac{1}{p}{\rm tr}(\vv{S}_{j_t,p}\vv{S}_{j_s,p}),&~\hbox{if both $j_t$ and $j_s$ are in $\{l_1,\ldots,l_k\}$},
\end{cases}
\]
and the reminder can be uniformly bounded.   
\end{corollary}
The proof of Corollary \ref{cor:B1} is quite similar as step (ii) by using (\ref{eq:A.2}) and taking conditional expectations repeatedly. Thus,  we omit the details to save space.

With the help of Lemma \ref{lem:B}, we are now ready to prove Lemma \ref{lem:detG}.  
\begin{proof}[Proof of Lemma \ref{lem:detG}]
Without loss of generality, we assume $(i_1,\ldots,i_{k})=(1,\ldots,k)$. Recall that
\[
{\rm det}(\hat{\vv{G}}_{1,\ldots,k})=\sum_{(j_1,\ldots,j_{k})} {\rm sgn}\left(\begin{smallmatrix}1&2&\cdots&k\\j_1&j_2&\cdots&j_{k}\end{smallmatrix}\right) \hat{G}_{1,j_1}\cdots\hat{G}_{k,j_{k}},
\]
where ${\rm sgn}\left(\begin{smallmatrix}1&2&\cdots&k\\j_1&j_2&\cdots&j_{k}\end{smallmatrix}\right)$ is the sign of permutation $\left(\begin{smallmatrix}1&2&\cdots&k\\j_1&j_2&\cdots&j_{k}\end{smallmatrix}\right)$. 

It suffices to show that for any permutation $(j_1,\ldots,j_{k})$,
\begin{equation}\label{eq:B.1}
\mathbb{E}\left[\hat{G}_{1,j_1}\cdots\hat{G}_{k,j_{k}}\right]=G_{1,j_1}\cdots G_{k,j_{k}}+O(\frac{1}{p^2}),
\end{equation}
where the reminder can be uniformly bounded. 

Note that the arrangement of $\left(\begin{smallmatrix}1&2&\cdots&k\\j_1&j_2&\cdots&j_{k}\end{smallmatrix}\right)$ corresponds uniquely to a permutation of $\{1,\ldots,k\}$. So, it can be written as a composition of several disjoint cyclic permutations. Since indices in distinct cyclic permutations are disjoint, the corresponding cyclic products $\hat{G}_{l_1,l_2}\hat{G}_{l_2,l_3}\cdot\hat{G}_{l_{m-1},l_m}\cdot\hat{G}_{l_m,l_1}$ must be independent. Thus, the expectation in the left-hand side of (\ref{eq:B.1}) can be written as a product of expectations of cyclic productions:
\[
\mathbb{E}\left[\hat{G}_{l_1,l_2}\hat{G}_{l_2,l_3}\cdot\hat{G}_{l_{m-1},l_m}\cdot\hat{G}_{l_m,l_1}\right].
\]
Then, by Lemma \ref{lem:B}, Equation (\ref{eq:B.1}) must hold.   
\end{proof}

The previous discussion focuses mainly on the expectation of cyclic products of $\hat{G}_{ij}$'s. The next lemma shows the uniform boundedness of their variances. 
\begin{lem}\label{lem:B2}
    Suppose that Assumptions \ref{assm:2.1}---\ref{assm:2.4} hold.
Let $m$ be a positive integer and $(j_1,\ldots,j_m)\in \{1,\ldots,q\}$ be a combination of distinct indices. Then, there exists a positive constant dependent only on $m$ such that
\begin{equation}\label{eq:B4}
    {\rm Var}(p\hat{G}_{j_1,j_2}\cdots \hat{G}_{j_m,j_1}) \leq C.
\end{equation}
Consequently, $\{{\rm Var}(p{\rm det}(\hat{G}(j_1,\ldots,j_m)))\}$ are uniformly bounded by certain constant dependent only on $m$. 
\end{lem}
\begin{proof}
In general, the proof is based on results in Section \ref{app:A} and the law of total variance:
\[
{\rm Var}(X) = \E[{\rm Var}(X|\mathcal{F})] + {\rm Var}(\E[X|\mathcal{F}]),
\]
where $X$ is a random variable and $\mathcal{F}$ is a $\sigma$-field.
To avoid tedious but lengthy calculation, we only take ${\rm Var}(p\hat{G}_{12}\hat{G}_{23}\hat{G}_{31})$ as an example. 

Let $W_p = p\hat{G}_{12}\hat{G}_{23}\hat{G}_{31}$. By the law of total variance, we have 
\begin{align*}
    {\rm Var}(W_p) & = \E[{\rm Var}(W_p|\vv{S}_{1,p},\vv{S}_{2,p})] + {\rm Var}(\E[W_p|\vv{S}_{1,p},\vv{S}_{2,p}]). 
\end{align*}
On one hand, by (\ref{eq:A.5}),
\begin{align*}
    &{\rm Var}\left(\frac{1}{p}{\rm tr}(\vv{S}_{3,p}\vv{S}_{2,p}){\rm tr}(\vv{S}_{3,p}\vv{S}_{1,p})|\vv{S}_{2,p},\vv{S}_{1,p}\right) \\
    =& 4 c_{3,p} \left(\frac{1}{p}{\rm tr}(\vv{S}_{1,p}\vv{\Sigma}_3)\right) \left(\frac{1}{p}{\rm tr}(\vv{S}_{2,p}\vv{\Sigma}_3)\right)\mu(\vv{S}_{1,p},\vv{S}_{2,p}|\vv{\Sigma}_3) + o(1),
\end{align*}
where 
\[
\left(\frac{1}{p}{\rm tr}(\vv{S}_{1,p}\vv{\Sigma}_3)\right),~~~ \left(\frac{1}{p}{\rm tr}(\vv{S}_{2,p}\vv{\Sigma}_3)\right)~~~\hbox{and}~~~\mu(\vv{S}_{1,p},\vv{S}_{2,p}|\vv{\Sigma}_3)
\]
can be uniformly bounded due to Assumptions \ref{assm:2.2}---\ref{assm:2.4}. Hence,
\[
{\rm Var}(W_p|\vv{S}_{1,p},\vv{S}_{2,p})= \left(\frac{1}{p}{\rm tr}(\vv{S}_{1,p}\vv{S}_{2,p})\right)^2\left[4 c_{3,p} \left(\frac{1}{p}{\rm tr}(\vv{S}_{1,p}\vv{\Sigma}_3)\right) \left(\frac{1}{p}{\rm tr}(\vv{S}_{2,p}\vv{\Sigma}_3)\right)\mu(\vv{S}_{1,p},\vv{S}_{2,p}|\vv{\Sigma}_3) + o(1)\right]
\]
can also be uniformly bounded. 

On the other hand, according to Corollary \ref{cor:B1},
\[
\E[W_p|\vv{S}_{1,p},\vv{S}_{2,p}] = \frac{1}{p^2}{\rm tr}(\vv{S}_{1,p}\vv{S}_{2,p}) {\rm tr}(\vv{S}_{2,p}\vv{\Sigma}_3) {\rm tr}(\vv{S}_{1,p}\vv{\Sigma}_3) + O(\frac{1}{n_3}) = W_{1,p}+ O(\frac{1}{n_3})
\]
Hence, ${\rm Var}(\E[W_p|\vv{S}_{1,p},\vv{S}_{2,p}]) = {\rm Var}(W_{1,p}) + o(1)$.
Then, we repeatedly use the law of total variance to have
\[
{\rm Var}(W_{1,p}) = \E[{\rm Var}(W_{1,p}|\vv{S}_{1,p})] + {\rm Var}(\E[W_{1,p}|\vv{S}_{1,p}]),
\]
where, by Equations (\ref{eq:A.2}) and (\ref{eq:A.5}), we have
\begin{align*}
    \E[W_{1,p}|\vv{S}_{1,p}] & = \left(\frac{1}{p} {\rm tr}(\vv{S}_{1,p}\vv{\Sigma}_3){\rm tr}(\vv{S}_{1,p}\vv{\Sigma}_2)\right)\left(\frac{1}{p}{\rm tr}(\vv{\Sigma}_2\vv{\Sigma}_3)\right) + O(\frac{1}{n_2}) \\
    {\rm Var}(W_{1,p}|\vv{S}_{1,p}) & = 4c_{2,p}\left(\frac{1}{p}{\rm tr}(\vv{S}_{1,p}\vv{\Sigma}_3)\right)^2 \left(\frac{1}{p}{\rm tr}(\vv{S}_{1,p}\vv{\Sigma}_2)\right)\left(\frac{1}{p}{\rm tr}(\vv{\Sigma}_{2}\vv{\Sigma}_3)\right)\mu(\vv{S}_{1,p},\vv{\Sigma}_3|\vv{\Sigma}_2) + O(\frac{1}{p}).
\end{align*}
Therefore, we conclude that ${\rm Var}(W_{1,p})$ is uniformly bounded. Combining discussion above, we also see that ${\rm Var}(\E[W_p|\v{S}_{1,p},\vv{S}_{2,p}])$ is uniformly bounded. 
\end{proof}

\section{Proof of Theorem \ref{thm:AN}}\label{app:C}
We derive the asymptotic normaility of $\hat{M}_p^{(d_0+1)}$ by considering its H\'{a}jeck projection onto the space 
\[
\mathcal{S}_p:=\left\{\sum_{i=1}^q g_{i,p}(\vv{X}_{i,p}):\mathbb{E}[g_{i,p}(\vv{X}_{i,p})^2]<\infty,~i=1,\ldots,q\right\}.
\]
Specifically, the H\'{a}jeck projection
\begin{align*}
   \tilde{M}_p^{(d_0+1)} &:= \E\left\{\hat{M}_p^{(d_0+1)}-\E[\hat{M}_p^{(d_0+1)}]|\mathcal{S}_p\right\}\\
   &= \frac{(d_0+1)}{q} \sum_{i=1}^q \E\left\{\hat{M}_p^{(d_0+1)} - \E[\hat{M}_p^{(d_0+1)}]|\vv{X}_{i,p}\right\}\\
   &=: \frac{(d_0+1)}{q}\sum_{i=1}^q h_{1,i}(\vv{X}_{i,p})
\end{align*}
is in fact a sum of independent random variables, where we denote 
\begin{align*}
h_{1,i}(\vv{X}_{i,p})&=\binom{q-1}{d_0}^{-1}\sum_{(i_2,\ldots,i_{d_0+1})\neq i} \bigg\{\mathbb{E}\left[{\rm det}(\hat{\vv{G}}(i,i_2,\ldots,i_{d_0+1}))|\vv{X}_{i,p}\right]\\
&~~~~~~~~~~~~~~~~~~~~~~~~~~~~~~~~~~~~~~~~~~-\E[\hat{\vv{G}}(i,i_2,\ldots,i_{d_0+1})]\bigg\}.
\end{align*}
We conclude that $\tilde{M}_p^{(d_0+1)}$ is a ``good'' approximation to $\hat{M}_p^{(d_0+1)}$ in the sense that
\[
\frac{p(\hat{M}_p^{(d_0+1)} - \E[\hat{M}_p^{(d_0+1)}])}{{\rm Var}(\hat{M}_p^{(d_0+1)})^{1/2}} - \frac{p\tilde{M}_p^{(d_0+1)}}{{\rm Var}(\tilde{M}_p^{(d_0+1)})^{1/2}} \overset{p.}{\to} 0.
\]
In fact, according to Theorem 11.2 in \cite{vander2000},the conclusion holds if 
\[
{\rm Var}(\tilde{M}_p^{(d_0+1)})/{\rm Var}(\hat{M}_p^{(d_0+1)}) \to 1.
\]
To verify it, recall in Lemma \ref{lem:B2} that ${\rm Var}(p{\rm det}(\hat{G}(i_1,\ldots,i_{d_0+1})))< C$ for certain constant $C$ dependent only on $d_0$. Then, for any $I=(i_1,\ldots,i_{d_0+1})$ and $J=(j_1,\ldots,j_{d_0+1})$, 
\[
|{\rm Cov}(p{\rm det}(\hat{\vv{G}}(I)),p{\rm det}(\hat{\vv{G}}(J)))|\leq {\rm Var}(p{\rm det}(\hat{G}(I))^{1/2}{\rm Var}(p{\rm det}(\hat{G}(J))^{1/2} \leq C,
\]
and
\begin{align*}
    {\rm Var}(p\hat{M}_p^{(d_0+1)}) & = \binom{q}{d_0+1}^{-2} \sum_{\substack{I=(i_1,\ldots,i_{d_0+1})\\J=(j_1,\ldots,j_{d_0+1}))}} {\rm Cov}(p{\rm det}(\hat{\vv{G}}(I)),p{\rm det}(\hat{\vv{G}}(J))) \\
    & = \binom{q}{d_0+1}^{-2} \sum_{k=1}^{d_0+1} \sum_{I,J: |I\cap J| =k} {\rm Cov}(p{\rm det}(\hat{\vv{G}}(I)),p{\rm det}(\hat{\vv{G}}(J))) \\
    & = \frac{(d_0+1)^2}{q^2}\sum_{i=1}^q \binom{q-1}{d_0}^{-2}\sum_{I\cap J =\{i\}} {\rm Cov}(p{\rm det}(\hat{\vv{G}}(I)),p{\rm det}(\hat{\vv{G}}(J))) + o(\frac{1}{q})\\
    & = \frac{(d_0+1)^2}{q^2}\sum_{i=1}^q {\rm Var}(p h_{1,i}(\vv{X}_{i,p})) + o(\frac{1}{q}) \\
    & = {\rm Var}(p\tilde{M}_p^{(d_0+1)}) + o(\frac{1}{q}).
\end{align*}
Hence, to prove Theorem \ref{thm:AN}, it suffices to establish asymptotic normality of $p\tilde{M}_p^{(d_0+1)}$. 

According to Corollary \ref{cor:B1}, we have
\begin{equation}\label{eq:C.1}
\mathbb{E}\left[{\rm det}(\hat{\vv{G}}(i,i_2,\ldots,i_{d_0+1}))|\vv{X}_{i,p}\right]={\rm det}(\tilde{\vv{G}}(i,i_2,\ldots,i_{d_0+1}))+O(\frac{1}{p^2}),
\end{equation}
in which $\tilde{\vv{G}}(i,i_2,\ldots,i_{d_0+1})$ is the modified Gram matrix generated by $\vv{S}_{i,p}$ and $\{\vv{\Sigma}_{i_2},\ldots,\vv{\Sigma}_{i_{d_0+1}}\}$, i.e.,
\begin{equation}\label{eq:C.2}
\tilde{\vv{G}}(i,i_2,\ldots,i_{d_0+1}):=\begin{pmatrix}
            \tilde{G}_{ii}&\tilde{G}_{i,i_2}&\cdots&\tilde{G}_{i,i_{d_0+1}}\\
            \tilde{G}_{i_2,i}&G_{i_2,i_2}&\cdots&\mu_{i_2,i_{d_0+1}}\\
            \vdots&\vdots&\ddots&\vdots\\
            \tilde{G}_{i_{K+1},i}&G_{i_{K+1},i_2}&\cdots&G_{i_{d_0+1}i_{d_0+1}}
\end{pmatrix}
\end{equation}
with $\tilde{G}_{ii}=\hat{G}_{ii}$ and $\tilde{G}_{ij}=\tilde{G}_{ji}=\frac{1}{p}{\rm tr}(\vv{S}_{ip}\vv{\Sigma}_j)$ for $j\neq i$. It can be seen that only the $(1,1)$th entry is modified by $\hat{G}_{ii}$ to ensure an asymptotic unbiased estimation of $G_{ii}$. In addition, it also holds that
\begin{equation}\label{eq:C.3}
\E[\tilde{\vv{G}}(i,i_2,\ldots,i_{d_0+1})] = \E[\hat{\vv{G}}(i,i_2,\ldots,i_{d_0+1})]+O(\frac{1}{p^2}).
\end{equation}
It should be noticed that all reminders of order $O(1/p^2)$ in (\ref{eq:C.1}) and (\ref{eq:C.3}) can be bounded uniformly due to Assumption \ref{assm:2.2}.  
Thus, we conclude that
\begin{equation}
\begin{split}
h_{1,i}(\vv{X}_{i,p}) &= \binom{q-1}{d_0}^{-1}\sum_{(i_2,\ldots,i_{d_0+1})\neq i} \bigg\{{\rm det}(\tilde{\vv{G}}(i,i_2,\ldots,i_{d_0+1}))\\
&~~~~~~~~~~~~~~~~~~~~~~~~~~~~~~~~~-\E[{\rm det}(\tilde{\vv{G}}(i,i_2,\ldots,i_{d_0+1}))]\bigg\}+O(\frac{1}{p^2}).
\end{split}
\end{equation}

Let
\begin{equation}\label{eq:Hip}
H_{i,p} = \binom{q-1}{d_0}^{-1}\sum_{(i_2,\ldots,i_{d_0+1})\neq i}{\rm det}(\tilde{\vv{G}}(i,i_2,\ldots,i_{d_0+1})).
\end{equation}
Then, we have ${\rm Var}(p h_{1,i}(\vv{X}_{i,p})) = {\rm Var}(p H_{i,p})$. In what follows, we are going to calculate the leading term of the variance of $H_{i,p}$ with the help of results in Section \ref{app:A}. 

Fixing order of $(i_2,\ldots,i_{d_0+1})$, we expand ${\rm det}(\tilde{\vv{G}}(i,i_2,\ldots,i_{d_0+1}))$ :
\begin{align*}
&{\rm det}(\tilde{\vv{G}}(i,i_2,\ldots,i_{d_0+1})) \\
=&\tilde{G}_{ii}\cdot {\rm det}(\vv{G}(i_2,\ldots,i_{d_0+1}))+ \sum_{j=2}^{d_0+1}(-1)^{j+1}\tilde{G}_{i_j,i}{\rm det}(\tilde{\vv{G}}_{i_j}(i_2,\ldots,i_{d_0+1}))\\
=&\tilde{G}_{ii}\cdot {\rm det}(\vv{G}(i_2,\ldots,i_{d_0+1})) + \sum_{j,l=2}^{d_0+1}(-1)^{j+l+1} \tilde{G}_{i_j,i}\tilde{G}_{i,i_l}  {\rm det}(\vv{G}_{i_j|i_l}(i_2,\ldots,i_{d_0+1}))\\
=&\tilde{G}_{ii}\cdot {\rm det}(\vv{G}(i_2,\ldots,i_{d_0+1}))-\sum_{j=2}^{d_0+1}\tilde{G}_{i_j,i}^2 {\rm det}(\vv{G}_{i_j|i_j}(i_2,\ldots,i_{d_0+1}))\\
+&\sum_{2\leq j\neq l\leq d_0+1} (-1)^{j+l+1}\tilde{G}_{i_j,i}\tilde{G}_{i,i_l}{\rm det}(\vv{G}_{i_j|i_l}(i_2,\ldots,i_{d_0+1}))
\end{align*}
where $\tilde{\vv{G}}_{i_j}(i_2,\ldots,i_{d_0+1})$ is a $d_0\times d_0$ matrix obtained by removing the first column and the $j$th row of $\tilde{\vv{G}}(i,i_2,\ldots,i_{d_0+1})$, 
and $\vv{G}_{i_j|i_l}(i_2,\ldots,i_{d_0+1})$ is derived by deleting the row with index $i_j$ and the column with index $i_l$ from $\vv{G}(i_2,\ldots,i_{d_0+1})$.

Let $A_n^k = n(n-1)\cdots(n-k+1)$. Then, we have
\begin{equation}\label{eq:H}
\begin{split}
&H_{i,p}=\binom{q-1}{K}^{-1}\sum_{(i_2,\ldots,i_{d_0+1})\neq i} {\rm det}(\tilde{\vv{G}}(i,i_2,\ldots,i_{d_0+1}))\\
&=\left(A_{q-1}^{d_0}\right)^{-1}\sum_{i_2,\cdots,i_{d_0+1}\neq i} {\rm det}(\tilde{\vv{G}}(i,i_2,\ldots,i_{K+1}))\\
&=\left(A_{q-1}^{d_0}\right)^{-1}\sum_{i_2,\cdots,i_{d_0+1}\neq i}  \bigg\{\tilde{G}_{ii}{\rm det}(\vv{G}(i_2,\ldots,i_{K+1}))\\
&~~~~~~~~~~~~~~~~~~~~~~~~~~~~~~~~~+\sum_{j,l=2}^{d_0+1}(-1)^{j+l+1} \tilde{G}_{i,i_{j}}\tilde{G}_{i_l,i}{\rm det}(\vv{G}_{i_j|i_l}(i_2,\ldots,i_{d_0+1}))
\bigg\}\\
&=\tilde{G}_{ii}\left\{\binom{q-1}{d_0}^{-1}\sum_{(i_2,\ldots,i_{d_0+1})\neq i} {\rm det}(\vv{G}(i_2,\ldots,i_{d_0+1}))\right\}\\
&-\frac{1}{q-1}\sum_{r\neq i} \hat{\mu}_{i,r}^2\cdot\left\{\left(A_{q-2}^{d_0-1}\right)^{-1} \sum_{j=2}^{d_0+1}\sum_{\substack{i_2,\ldots,i_{d_0+1}\neq i\\i_j = r}}{\rm det}(\vv{G}_{r|r}(i_2,\ldots,i_{d_0+1}))\right\}\\
+&\frac{1}{(q-1)(q-2)}\sum_{r\neq s \neq i} \tilde{G}_{i,r}\tilde{G}_{i,s}\left\{\left(A_{q-3}^{d_0-2}\right)^{-1}\sum_{2\leq j\neq l\leq d_0+1}(-1)^{j+l+1}\sum_{\substack{i_2,\ldots,i_{d_0+1}\neq i\\ i_j = r, i_l =s}}{\rm det}(\vv{G}_{r|s}(i_2,\ldots,i_{d_0+1}))\right\} \\
& = : \tilde{G}_{ii} \alpha_{i,q} - \frac{1}{q-1}\sum_{r\neq i} \tilde{G}_{i,r}^2 \beta_{ir,q} + \frac{1}{(q-1)(q-2)}\sum_{r\neq s\neq i} \tilde{G}_{i,r}\tilde{G}_{i,s} \gamma_{irs,q},
\end{split}
\end{equation}
where
\begin{align}
\alpha_{i,q}&:=\binom{q-1}{d_0}^{-1}\sum_{i_2,\ldots,i_{d_0+1}\neq i} {\rm det}(\vv{G}(i_2,\ldots,i_{d_0+1})),\\
\beta_{ir,q}&:=\left(A_{q-2}^{d_0-1}\right)^{-1} \sum_{j=2}^{d_0+1}\sum_{\substack{i_2,\ldots,i_{d_0+1}\neq i\\i_j = r}}{\rm det}(\vv{G}_{r|r}(i_2,\ldots,i_{d_0+1})),\\
\gamma_{irs,q}&:=\left(A_{q-3}^{d_0-2}\right)^{-1}\sum_{2\leq j\neq l\leq d_0+1}(-1)^{j+l+1}\sum_{\substack{i_2,\ldots,i_{d_0+1}\neq i\\ i_j = r, i_l =s}}{\rm det}(\vv{G}_{r|s}(i_2,\ldots,i_{d_0+1})).
\end{align}
Multiplying $H_{i,p}$ by $p$, we find that
\begin{align*}
{\rm Var}(p H_{i,p})&={\rm Var}(p\tilde{G}_{ii}) \alpha_{i,q}^2+\frac{1}{(q-1)^2}\sum_{r,s\neq i} {\rm Cov}(p\tilde{G}_{ir}^2,p\tilde{G}_{is}^2)\beta_{ir,q}\beta_{is,q}\\
&+\frac{1}{(q-1)^2(q-2)^2}\sum_{r\neq s;r'\neq s'}{\rm Cov}(p\tilde{G}_{ir}\tilde{G}_{is},p\tilde{G}_{ir'}\tilde{G}_{is'})\cdot \gamma_{irs,q}\gamma_{ir's',q}\\
&-\frac{2}{q-1}\sum_{r\neq i}{\rm cov}(p\tilde{G}_{ii},p\tilde{G}_{ir}^2)\alpha_{i,q}\beta_{ir,q}\\
&+\frac{2}{(q-1)(q-2)}\sum_{r\neq s\neq i}{\rm cov}(p\tilde{G}_{ii},p\tilde{G}_{ir}\tilde{G}_{is})\alpha_{i,q}\gamma_{irs,q}\\
&-\frac{2}{(q-1)^2(q-2)}\sum_{r\neq i}\sum_{r'\neq s'\neq i}{\rm cov}(p\tilde{G}_{ir}^2,p\tilde{G}_{ir'}\tilde{G}_{is'})\beta_{ir;q}\gamma_{ir's',q}.
\end{align*}

Recall Equations (\ref{eq:A.4}), (\ref{eq:A.5}) and (\ref{eq:A.6}), 
we have
\begin{equation}\label{eq:vr}
\begin{split}
{\rm Var}(p\tilde{G}_{ii})&= {\rm Var}\left({\rm tr}(\vv{S}_{i,p}^2) - c_{i,p}\frac{1}{p}{\rm tr}(\vv{S}_{i,p})^2\right) + O(\frac{1}{p})\\
&=4c_{i,p}^2G_{ii}^2+4c_{i,p}\mu(\vv{\Sigma}_i,\vv{\Sigma}_i|\vv{\Sigma}_i)+O(\frac{1}{p}),\\
{\rm Cov}(p\tilde{G}_{ir}^2,p\tilde{G}_{is}^2)&=4c_{i,p}G_{ir}G_{is}\mu(\vv{\Sigma}_r,\vv{\Sigma}_s|\vv{\Sigma}_i)+O(\frac{1}{p}),\\
{\rm Cov}(p\tilde{G}_{ii},p\tilde{G}_{ir}^2)&=4c_{i,p}G_{ir}\mu(\vv{\Sigma}_i,\vv{\Sigma}_r|\vv{\Sigma}_i)+O(\frac{1}{p}),\\
{\rm Cov}(p\tilde{G}_{ii},p\tilde{G}_{ir}\tilde{G}_{is})&=2c_{i,p}G_{ir}\mu(\vv{\Sigma}_i,\vv{\Sigma}_s|\vv{\Sigma}_i)\\
&+2c_{i,p}G_{is}\mu(\vv{\Sigma}_i,\vv{\Sigma}_r|\vv{\Sigma}_i)+O(\frac{1}{p}),\\
{\rm Cov}(p\tilde{G}_{ir}^2,p\tilde{G}_{ir'}\tilde{G}_{is'})&=2c_{i,p}G_{ir}G_{ir'}\mu(\vv{\Sigma}_r,\vv{\Sigma}_s'|\vv{\Sigma}_i)\\
&+2c_{i,p}G_{ir}G_{is'}\mu(\vv{\Sigma}_r,\vv{\Sigma}_r'|\vv{\Sigma}_i)+O(\frac{1}{p}),\\
{\rm Cov}(p\tilde{G}_{ir}\tilde{G}_{is},p\tilde{G}_{ir'}\tilde{G}_{is'})&=c_{i,p}G_{ir}G_{ir'}\mu(\vv{\Sigma}_s,\vv{\Sigma}_s'|\vv{\Sigma}_i)\\
&+c_{i,p}G_{ir}G_{is'}\mu(\vv{\Sigma}_s,\vv{\Sigma}_r'|\vv{\Sigma}_i)\\
&+c_{i,p}G_{is}G_{ir'}\mu(\vv{\Sigma}_r,\vv{\Sigma}_{s'}|\vv{\Sigma}_i)\\
&+c_{i,p}G_{is}G_{is'}\mu(\vv{\Sigma}_r,\vv{\Sigma}_{r'}|\vv{\Sigma}_i)+O(\frac{1}{p}),
\end{split}
\end{equation}
where $\mu(\cdot,\cdot|\vv{\Sigma}_i)$ follows definition (\ref{eq:A.1}). Then,
\begin{align*}
{\rm Var}(p H_{i,p})&= \alpha_{i,q}^2\left\{4c_{i,p}^2G_{ii}^2+4c_{i,p}\mu(\vv{\Sigma}_i,\vv{\Sigma}_i|\vv{\Sigma}_i)\right\}\\
&+\frac{1}{(q-1)^2}\sum_{r,s\neq i} 4c_{i,p}G_{ir}G_{is}\beta_{ir,q}\beta_{is,q}\mu(\vv{\Sigma}_r,\vv{\Sigma}_s|\vv{\Sigma}_i)\\
&+\frac{1}{(q-1)^2(q-2)^2}\sum_{r\neq s;r'\neq s'} \gamma_{irs,q}\gamma_{ir's',q}\bigg\{c_{i,p}G_{ir}G_{ir'}\mu(\vv{\Sigma}_s,\vv{\Sigma}_s'|\vv{\Sigma}_i)\\
&~~~~~~~~~~~~~~~~~~~~~~~~~~~~~~~~~~~~~~~~~~~~~~~~~~+c_{i,p}G_{ir}G_{is'}\mu(\vv{\Sigma}_s,\vv{\Sigma}_r'|\vv{\Sigma}_i)\\
&~~~~~~~~~~~~~~~~~~~~~~~~~~~~~~~~~~~~~~~~~~~~~~~~~~+c_{i,p}G_{is}G_{ir'}\mu(\vv{\Sigma}_r,\vv{\Sigma}_s'|\vv{\Sigma}_i)\\
&~~~~~~~~~~~~~~~~~~~~~~~~~~~~~~~~~~~~~~~~~~~~~~~~~~+c_{i,p}G_{is}G_{is'}\mu(\vv{\Sigma}_r,\vv{\Sigma}_r'|\vv{\Sigma}_i)\bigg\}\\
&-\frac{2}{q-1}\sum_{r\neq i}4c_{i,p}G_{ir}\alpha_{i,q}\beta_{ir,q}\mu(\vv{\Sigma}_i,\vv{\Sigma}_r|\vv{\Sigma}_i)\\
&+\frac{2}{(q-1)(q-2)}\sum_{r\neq s\neq i}\alpha_{i,q}\gamma_{irs,q}\bigg\{2c_{i,p}G_{ir}\mu(\vv{\Sigma}_i,\vv{\Sigma}_s|\vv{\Sigma}_i)\\
&~~~~~~~~~~~~~~~~~~~~~~~~~~~~~~~~~~~~~~~~~~~~+2c_{i,p}G_{is}\mu(\vv{\Sigma}_i,\vv{\Sigma}_r|\vv{\Sigma}_i)\bigg\}\\
&-\frac{2}{(q-1)^2(q-2)}\sum_{r\neq i}\sum_{r'\neq s'\neq i}\beta_{ir;q}\gamma_{ir's',q}\bigg\{2c_{i,p}G_{ir}G_{ir'}\mu(\vv{\Sigma}_r,\vv{\Sigma}_{s'}|\vv{\Sigma}_i)\\
&~~~~~~~~~~~~~~~~~~~~~~~~~~~~~~~~~~~~~~~~~~~~~~~~~~+2c_{i,p}G_{ir}G_{is'}\mu(\vv{\Sigma}_r,\vv{\Sigma}_{r'}|\vv{\Sigma}_i)\bigg\}  + O(\frac{1}{p})\\
& = \alpha_{i,q}^2\left\{4c_{i,p}^2 G_{ii}^2\right\} + 4c_{i,p}\mu(\alpha_{i,q}\vv{\Sigma}_i,\alpha_{i,q}\vv{\Sigma}_i|\vv{\Sigma}_i) \\
& + 4c_{i,p}\mu\left(\frac{1}{q-1}\sum_{r\neq i} G_{ir}\beta_{ir,q}\vv{\Sigma}_r, \frac{1}{q-1}\sum_{s\neq i} G_{is}\beta_{is,q}\vv{\Sigma}_s\bigg|\vv{\Sigma}_i\right)\\
& + 4c_{i,p}\mu\left(\frac{1}{(q-1)(q-2)}\sum_{r\neq s\neq i} G_{ir}\gamma_{irs,q}\vv{\Sigma}_s, \frac{1}{(q-1)(q-2)}\sum_{r'\neq s'\neq i} G_{ir'}\gamma_{ir's',q}\vv{\Sigma}_{s'}\bigg|\vv{\Sigma}_i\right)\\
& -2\cdot 4c_{i,p}\mu\left(\alpha_{i,q}\vv{\Sigma}_i,\frac{1}{q-1}\sum_{r\neq i} G_{ir}\beta_{ir,q}\vv{\Sigma}_r\bigg|\vv{\Sigma}_i\right)\\
& +2\cdot 4c_{i,p}\mu\left(\alpha_{i,q}\vv{\Sigma}_i,\frac{1}{(q-1)(q-2)}\sum_{r\neq s\neq i} G_{ir}\gamma_{irs,q}\vv{\Sigma}_s\bigg|\vv{\Sigma}_i\right)\\
& - 2\cdot 4c_{i,p}\mu\left(\frac{1}{q-1}\sum_{r\neq i} G_{ir}\beta_{ir,q}\vv{\Sigma}_r, \frac{1}{(q-1)(q-2)}\sum_{r'\neq s'\neq i} G_{ir'}\gamma_{ir's',q}\vv{\Sigma}_{s'}\bigg|\vv{\Sigma}_i\right) + O(\frac{1}{p}),
\end{align*}
where we use the fact that $\gamma_{irs,q} = \gamma_{isr,q}$ and $\mu(\cdot,\cdot|\vv{\Sigma}_i)$
is bilinear. Further, we define
\begin{equation}
    \vv{R}_i = \alpha_{i,q}\vv{\Sigma}_i - \frac{1}{q-1}\sum_{r\neq i} G_{ir}\beta_{ir,q}\vv{\Sigma}_r + \frac{1}{(q-1)(q-2)}\sum_{r\neq s\neq i} G_{ir}\gamma_{irs,q}\vv{\Sigma}_{s}.
\end{equation}
By using similar technique as in (\ref{eq:H}), we find that
\begin{align*}
\vv{R}_i=&\left(A_{q-1}^{d_0}\right)^{-1}\sum_{i_2,\cdots,i_{d_0+1}\neq i} {\rm det}
\begin{pmatrix}
    \vv{\Sigma}_i & \vv{\Sigma}_{i_2} &\cdots&\vv{\Sigma}_{i_{d_0+1}} \\
    G_{i_2,i} & G_{i_2,i_2} & \cdots & G_{i_2,i_{d_0+1}} \\
    \vdots & \vdots &\ddots& \vdots \\
    G_{i_{d_0+1},i} & G_{i_{d_0+1},i_2} & \cdots & G_{i_{d_0+1},i_{d_0+1}}
\end{pmatrix}\\
=&\binom{q-1}{d_0}^{-1}\sum_{(i_2,\ldots,i_{d_0+1})\neq i} {\rm det}
\begin{pmatrix}
    \vv{\Sigma}_i & \vv{\Sigma}_{i_2} &\cdots&\vv{\Sigma}_{i_{d_0+1}} \\
    G_{i_2,i} & G_{i_2,i_2} & \cdots & G_{i_2,i_{d_0+1}} \\
    \vdots & \vdots &\ddots& \vdots \\
    G_{i_{d_0+1},i} & G_{i_{d_0+1},i_2} & \cdots & G_{i_{d_0+1},i_{d_0+1}}
\end{pmatrix}.
\end{align*}
In addition, under the null hypothesis, any $(d_0+1)$ distinct covariance matrices are linearly dependent, thus all determinants in the equation about $\vv{R}_i$ degenerate to $\vv{O}$.

Meanwhile, we recall that
\begin{align*}
    \alpha_{i,q} &= \binom{q-1}{d_0}^{-1} \sum_{(i_2,\ldots,i_{d_0+1})\neq i} {\rm det}(\vv{G}(i_2,\ldots,i_{d_0+1})) \\
    & = \binom{q}{d_0}^{-1} \sum_{(i_1,\ldots,i_{d_0})} {\rm det}(\vv{G}(i_1,\ldots,i_{d_0})) + O(\frac{1}{q})\\
    & = M_p^{(d_0)} + O(\frac{1}{q}),
\end{align*}
where the reminder can be uniformly bounded by a constant independent of $i$.
Therefore, we have
\begin{align*}
{\rm Var}(pH_{i,p})&=\alpha_{i,q}^2\left\{4c_{i,p}^2G_{ii}^2\right\}
+4c_{i,p}\mu\left(\vv{R}_i,\vv{R}_i|\vv{\Sigma}_i\right) + O(\frac{1}{p})\\
& = \left(M_p^{(d_0)}\right)^2\left\{4c_{i,p}^2G_{ii}^2\right\}+ 4c_{i,p}\mu\left(\vv{R}_i,\vv{R}_i|\vv{\Sigma}_i\right) + O(\frac{1}{p} + \frac{1}{q})\\
& \geq \left(M_p^{(d_0)}\right)^2\left\{4c_{i,p}^2G_{ii}^2\right\} + O(\frac{1}{p} + \frac{1}{q}),
\end{align*}
since $\mu(\vv{R}_i,\vv{R}_i|\vv{\Sigma}_i) \geq 0$.  In particular, under $H_0$, $\mu(\vv{R}_i,\vv{R}_i|\vv{\Sigma}_i) = 0$ and in this case
\[
{\rm Var}(pH_{i,p}) = \left(M_p^{(d_0)}\right)^2\left\{4c_{i,p}^2G_{ii}^2\right\} + O(\frac{1}{p}+\frac{1}{q}).
\]

Recall that $\tilde{M}_p^{(d_0+1)} = \frac{(d_0+1)}{q}\sum_{i=1}^q h_{1,i}(\vv{X}_{i,p})$ is a sum of independent random variable and
\begin{align*}
    {\rm Var}(p\tilde{M}_p^{(d_0+1)}) & = \frac{(d_0+1)^2}{q^2} \sum_{i=1}^q {\rm Var}(ph_{1,i}(\vv{X}_{i,p})) 
     = \frac{(d_0+1)^2}{q^2} \sum_{i=1}^q {\rm Var}(pH_{i,p})\\
     & = \frac{(d_0+1)^2}{q^2} \sum_{i=1}^q \left\{\left(M_p^{(d_0)}\right)^2\left\{4c_{i,p}^2G_{ii}^2\right\}
+4c_{i,p}\mu\left(\vv{R}_i,\vv{R}_i|\vv{\Sigma}_i\right) + O(\frac{1}{p})\right\}\\
& = \frac{4(d_0+1)^2}{q^2}\left\{\left(M_p^{(d_0)}\right)^2 \beta_p + r_p^{(d_0)}\right\} + O(\frac{1}{qp})\\
& = \frac{1}{q}\left(\sigma_p^{(d_0)}\right)^2,
\end{align*}
where $\sigma_p^{(d_0)}$, $\beta_p$ and $r_p^{(d_0)}$ are quantities given in Theorem \ref{thm:AN}. Moreover, by applying similar but tedious calculation as in Section \ref{app:A}, we have
\[
\sum_{j=1}^q \E[|ph_{1,i}(\vv{X}_{i,p})|^4] = \sum_{j=1}^q \E[|pH_{i,p}|^4] + O(\frac{q}{p^2})= O(q) 
\]
so that the Lyaponuouv condition holds, that is,
\[
\sum_{j=1}^q \E[|ph_{1,i}(\vv{X}_{i,p})|^4]\bigg/\left(\sum_{j=1}^q {\rm Var}(ph_{1,i}(\vv{X}_{i,p}))\right)^2 = O(\frac{q}{q^2}) = O(\frac{1}{q}) = o(1).
\]
Hence, the Lyaponouv CLT suggests that
\[
\frac{p\tilde{M}_p^{(d_0+1)}}{{\rm Var}(\tilde{M}_p^{(d_0+1)})^{1/2}} = \frac{q^{1/2}p \tilde{M}_p^{(d_0+1)}}{\sigma_p^{(d_0)}} \overset{d.}{\to} \mathcal{N}(0,1).
\]
The proof is then completed.

\section{Proof of Theorem \ref{thm:var}}\label{app:E}

We first show that $\hat{M}_p^{(d_0)}-M_p^{(d_0)}$ converges to $0$ in probability. Recall Equation (\ref{eq:Uexp}) that
\[
\E[\hat{M}_p^{(d_0)}] = M_p^{(d_0)} + O(\frac{1}{p^2}). 
\]
It suffices to prove that ${\rm Var}(\hat{M}_p^{(d_0)}) = O(1/p^2)$. Observe that
\begin{align*}
    {\rm Var}(\hat{M}_p^{(d_0)}) & = \binom{q}{d_0}^{-2} \sum_{\substack{I =(i_1,\ldots,i_{d_0})\\ J =(j_1,\ldots,j_{d_0})}} {\rm Cov}({\rm det}(\hat{\vv{G}}(I),\hat{\vv{G}}(J)))
\end{align*}
According to Lemma \ref{lem:B2}, there exists a constant $C$ dependent only on $d_0$ such that 
\[
{\rm Var}(p{\rm det}(\hat{G}(I))) \leq C~~~~\hbox{for any}~I =(i_1,\ldots,i_{d_0}).
\]
It then follows that
\[
{\rm Cov}({\rm det}(\hat{G}(I)),{\rm det}(\hat{G}(J))) \leq \frac{1}{p^2} {\rm Cov}(p{\rm det}(\hat{G}(I)),p{\rm det}(\hat{G}(J))) \leq \frac{C}{p^2}
\]
for any $I$ and $J$. Consequently,
\[
 {\rm Var}(\hat{M}_p^{(d_0)})  \leq \binom{q}{d_0}^{-2} \sum_{\substack{I =(i_1,\ldots,i_{d_0})\\ J =(j_1,\ldots,j_{d_0})}} |{\rm Cov}({\rm det}(\hat{\vv{G}}(I),\hat{\vv{G}}(J)))| \leq \frac{C}{p^2}. 
\]
Hence, our conclusion holds.

Next, we prove that $\hat{\beta}_p-\beta_p$ converges in probability to zero. Observe that
\begin{align*}
    |\hat{\beta}_p -\beta_p| &\leq  \frac{1}{q}\sum_{j=1}^q c_{j,p}^2 |\hat{G}_{jj}^2 - G_{jj}^2| \\
    &=\frac{1}{q}\sum_{j=1}^q c_{j,p}^2 |\hat{G}_{jj}+ G_{jj}|\cdot|\hat{G}_{jj}-G_{jj}| \\
    & \leq \left(\max_{j} \left\{|\hat{G}_{jj}|+G_{jj}\right\}\right) \cdot \frac{1}{q}\sum_{j=1}^q c_{j,p}^2|\hat{G}_{jj} - G_{jj}|. 
\end{align*}
On one hand, according to variance calculation in (\ref{eq:vr}), we have
\begin{align*}
{\rm Var}(p\hat{G}_{ii})&= {\rm Var}\left({\rm tr}(\vv{S}_{i,p}^2) - c_{i,p}\frac{1}{p}{\rm tr}(\vv{S}_{i,p})^2\right) + O(\frac{1}{p})\\
&=4c_{i,p}^2G_{ii}^2+4c_{i,p}\mu(\vv{\Sigma}_i,\vv{\Sigma}_i|\vv{\Sigma}_i)+O(\frac{1}{p}),
\end{align*}
which suggests that 
\[
\max_{j }{\rm Var}(\hat{G}_{jj}) \leq \frac{C}{p^2} 
\]
for certain constant $C$ independent of $\{p,q,n_j\}$ due to Assumptions \ref{assm:2.2}---\ref{assm:2.4}. Hence,
\begin{align*}
    \frac{1}{q}\sum_{j=1}^q c_{j,p}^2 \E|\hat{G}_{jj}-G_{jj}| \leq \frac{1}{q}\sum_{j=1}^q c_{jp}^2 \frac{\sqrt{C}}{p} = O(\frac{1}{p}).
\end{align*}
It implies that
\begin{equation}\label{eq:vr1}
    \frac{1}{q}\sum_{j=1}^q c_{j,p}^2 |\hat{G}_{jj} - G_{jj}| = o_p(1). 
\end{equation}

On the other hand, notice that
\[
\max_{j}|\hat{G}_{jj}| \leq \max_j G_{jj} + \max_j |\hat{G}_{jj}-G_{jj}|
\]
and $\max_j G_{jj}$ is uniformly bounded due to Assumption \ref{assm:2.2}. It suffices to show that 
\begin{equation}\label{eq:vr2}
    \max_{j} |\hat{G}_{jj}-G_{jj}| = o_p(1).
\end{equation}
To see this, we observe that 
\begin{align*}
\P(|\hat{G}_{jj} - G_{jj}| >\varepsilon) 
    & \leq \frac{{\rm Var}(\hat{G}_{jj})}{\varepsilon^2} \leq \frac{C}{\varepsilon^2 p^2}
\end{align*}
for any $\varepsilon>0$. Then, we have
 \begin{align*}
 \P\left(\max_{j} |\hat{G}_{jj} - G_{jj}| > \varepsilon\right) & = \P\left(\bigcup_{j=1}^q \{|\hat{G}_{jj} -G_{jj}|> \varepsilon\}\right) \\
 & \leq \sum_{j=1}^q \P(|\hat{G}_{jj} - G_{jj}|> \varepsilon) \\
 & \leq \frac{qC^2}{\varepsilon^2 p^2} \to0,
 \end{align*}
 due to Assumption \ref{assm:2.4}. Combining (\ref{eq:vr1}) and (\ref{eq:vr2}), we see that
 \[
 |\hat{\beta}_p - \beta_p| \leq o(1) \cdot o(1) = o(1).
 \]
 In other words, $\hat{\beta}_p - \beta_p$ tends to zero in probability. 

Combing discussion above, we see that $(\hat{\sigma}_p^{(d_0)})^2$ is indeed consistent to $(\sigma_{0,p}^{(d_0)})^2$. Then, a direct application of Slutsky lemma implies the asymptotic normality of $q^{1/2}p\hat{M}_{p}^{(d_0+1)}/\hat{\sigma}_p^{(d_0)}$. The proof is completed. 

\section{Proof of Lemma \ref{lem:decomp}}\label{app:decomp}

First, we notice that $M^{(d)}_{p,K} >0$ since $\{\vv{\Sigma}_{j}\}_{j\leq q-K}$ are linearly independent.

Observe in (\ref{eq:e0}) that $\vv{\Sigma}_0$ is a linear combination of $\{\vv{\Sigma}_{j}\}_{j\leq q-K}$. Thus, we have $\vv{\Sigma}_0 \in \mathcal{H}_0$. 

Note that
\begin{align*}
    M^{(d)}_{p,K} \vv{\Sigma} & = \binom{q-K}{d}^{-1} \sum_{(i_1,\ldots,i_d)\in[q-K]} {\rm det}\begin{pmatrix}
    \vv{\Sigma} & \vv{\Sigma}_{i_1} & \cdots &\vv{\Sigma}_{i_d} \\
    0 &
    \langle\vv{\Sigma}_{i_1},\vv{\Sigma}_{i_1}\rangle&
    \cdots&
    \langle\vv{\Sigma}_{i_1},\vv{\Sigma}_{i_d}\rangle\\
    \vdots & \vdots& \ddots&\vdots\\
    0 &
    \langle\vv{\Sigma}_{i_d},\vv{\Sigma}_{i_1}\rangle&
    \cdots&
    \langle\vv{\Sigma}_{i_d},\vv{\Sigma}_{i_d}\rangle
    \end{pmatrix}
\end{align*}
Meanwhile, by Equation (\ref{eq:e0}) and (\ref{eq:eperp}), we also have 
\begin{equation}\label{eq:ee}
M^{(d)}_{p,K} \vv{\Sigma} = M^{(d)}_{p,K} \vv{\Sigma}_0 + M^{(d)}_{p,K}\vv{\Sigma}_\perp.
\end{equation}
It suffices to show that $\vv{\Sigma}_\perp\in\mathcal{H}_K^\perp$ by showing that for any $j\in\{1,\ldots,q-K\}$,
$
\langle \vv{\Sigma}_\perp,\vv{\Sigma}_{j}\rangle = 0.
$
In fact, for any $j=1,\ldots,q-K$ and $(i_1,\ldots,i_d)\in[q-K]$, the collection of $(d+1)$ matrices, $\{\vv{\Sigma}_{j},\vv{\Sigma}_{i_1},\ldots,\vv{\Sigma}_{i_d}\}$, must be linearly dependent, since $\mathcal{H}_K$ has dimension $d$. It then follows that
\[
{\rm det}\begin{pmatrix}
    \langle\vv{\Sigma}_{j},\vv{\Sigma} \rangle&  \langle\vv{\Sigma}_{j},\vv{\Sigma}_{i_1}\rangle & \cdots & \langle\vv{\Sigma}_{j},\vv{\Sigma}_{i_d} \rangle \\
    \langle\vv{\Sigma}_{i_1},\vv{\Sigma}\rangle &
    \langle\vv{\Sigma}_{i_1},\vv{\Sigma}_{i_1}\rangle&
    \cdots&
    \langle\vv{\Sigma}_{i_1},\vv{\Sigma}_{i_d}\rangle\\
    \vdots & \vdots& \ddots&\vdots\\
    \langle\vv{\Sigma}_{i_d},\vv{\Sigma}\rangle &
    \langle\vv{\Sigma}_{i_d},\vv{\Sigma}_{i_1}\rangle&
    \cdots&
    \langle\vv{\Sigma}_{i_d},\vv{\Sigma}_{i_d}\rangle
    \end{pmatrix} =0,
    \]
since rows of the determinant are linear dependent.
Thus, we have
\begin{align*}
M^{(d)}_{p,K}\langle \vv{\Sigma}_\perp,\vv{\Sigma}_{j}\rangle = \binom{q-K}{d}^{-1} \sum_{(i_1,\ldots,i_d)\in[q-K]} {\rm det}\begin{pmatrix}
    \langle\vv{\Sigma}_{j},\vv{\Sigma} \rangle&  \langle\vv{\Sigma}_{j},\vv{\Sigma}_{i_1}\rangle & \cdots & \langle\vv{\Sigma}_{j},\vv{\Sigma}_{i_d} \rangle \\
    \langle\vv{\Sigma}_{i_1},\vv{\Sigma}\rangle &
    \langle\vv{\Sigma}_{i_1},\vv{\Sigma}_{i_1}\rangle&
    \cdots&
    \langle\vv{\Sigma}_{i_1},\vv{\Sigma}_{i_d}\rangle\\
    \vdots & \vdots& \ddots&\vdots\\
    \langle\vv{\Sigma}_{i_d},\vv{\Sigma}\rangle &
    \langle\vv{\Sigma}_{i_d},\vv{\Sigma}_{i_1}\rangle&
    \cdots&
    \langle\vv{\Sigma}_{i_d},\vv{\Sigma}_{i_d}\rangle
    \end{pmatrix} =0,
\end{align*}
The conclusion then follows.

Finally, by (\ref{eq:Me}) and (\ref{eq:ee}), we have
\begin{align*}
    M^{(d)}_{p,K}(\vv{\Sigma}) & = \binom{q-K}{d}^{-1}\sum_{(i_1,\ldots,i_d)\in[q-K]} {\rm det}(\vv{G}(\vv{\Sigma};i_1,\ldots,i_d))\\
    &=\left\langle \vv{\Sigma},\binom{n}{d}^{-1}\sum_{(i_1,\ldots,i_d)} {\rm det}\begin{pmatrix}
    \vv{\Sigma} & \vv{\Sigma}_{i_1} & \cdots &\vv{\Sigma}_{i_d} \\
    \langle\vv{\Sigma}_{i_1},\vv{\Sigma}\rangle &
    \langle\vv{\Sigma}_{i_1},\vv{\Sigma}_{i_1}\rangle&
    \cdots&
    \langle\vv{\Sigma}_{i_1},\vv{\Sigma}_{i_d}\rangle\\
    \vdots & \vdots& \ddots&\vdots\\
    \langle\vv{\Sigma}_{i_d},\vv{\Sigma}\rangle &
    \langle\vv{\Sigma}_{i_d},\vv{\Sigma}_{i_1}\rangle&
    \cdots&
    \langle\vv{\Sigma}_{i_d},\vv{\Sigma}_{i_d}\rangle
    \end{pmatrix}\right\rangle \\
    & = \langle \vv{\Sigma}, M^{(d)}_{p,K}\vv{\Sigma}_\perp\rangle = M^{(d)}_{p,K} \langle\vv{\Sigma}_\perp,\vv{\Sigma}_\perp\rangle. 
\end{align*}
The proof is completed.

\section{Proof of Theorem \ref{thm:pw}}\label{app:D}   

According to Theorem \ref{thm:AN}, $q^{1/2}p(\hat{M}_p^{(d_0+1)}-M_p^{(d_0+1)})/\sigma_p^{(d_0)}\overset{d.}{\to} \mathcal{N}(0,1)$, where 
\[
(\sigma_p^{(d_0)})^2 = 4(d_0+1)^2\left\{(M_p^{(d_0)})^2\beta_p + r_p^{(d_0)}\right\}.
\]
We conclude that 
\begin{equation}\label{eq:Mapp}
    M_p^{(d_0)} = M_{p,K}^{(d_0)} + o(1)
\end{equation}
and
\begin{equation}\label{eq:Rapp}
    r_p^{(d_0)} = o(1)+ O(\frac{K}{q}) = o(1)
\end{equation}
under Assumptions \ref{assm:2.1} --- \ref{assm:2.5}. It immediately follows that $q^{1/2}p(\hat{M}_p^{(d_0+1)})/2(d_0+1)M_{p,K}^{(d_0)}\beta_p^{1/2}\overset{d.}{\to}\mathcal{N}(0,1)$. In what follows, we show Equations (\ref{eq:Mapp}) and (\ref{eq:Rapp}) hold. In fact, we let $J_0 =\{q-K+1,\ldots,q\}$. According to (\ref{eq:M}) and (\ref{eq:Mpk}), we find 
\begin{align*}
    M_p^{(d_0)} &= \binom{q}{d_0}^{-1}\sum_{I =(i_1,\ldots,i_{d_0})} {\rm det}(\vv{G}(I)) \\
    & = M_{p,K}^{(d_0)} + \binom{q}{d_0}^{-1}\sum_{l=1}^{\min\{K,d_0\}} \sum_{|I\cap J_0| = l} {\rm det}(\vv{G}(i_1,\ldots,i_{d_0})). 
\end{align*}
Due to Assumptions \ref{assm:2.2} --- \ref{assm:2.4}, there exists a constant $C'$ such that for any $I = (i_1,\ldots,i_{d_0})\in[q]$, 
\[
    0\leq {\rm det}(\vv{G}(I)) \leq C'.
\]
It then follows that for any $l=1,\ldots,{\min\{K,d_0\}} $,
\[
0\leq \sum_{|I\cap J_0| = l} {\rm det}(\vv{G}(I)) \leq C'\binom{K}{l}\binom{q-K}{d_0-l}. 
\]
Thus, by Assumption \ref{assm:2.5}, (\ref{eq:Mapp}) holds.

As for (\ref{eq:Rapp}), we observe that under Assumptions \ref{assm:2.2}---\ref{assm:2.4}, matrices
\[
{\rm det}\begin{pmatrix}
    \vv{\Sigma}_{i_1} & \vv{\Sigma}_{i_2} &\cdots & \vv{\Sigma}_{i_{d_0+1}}\\
    G_{i_2,i_1} & G_{i_2,i_2} & \cdots & G_{i_2,i_{d_0+1}}\\
    \vdots&\vdots&\ddots&\vdots \\
    G_{i_{d_0+1},i_1} & G_{i_{d_0+1},i_2} &\cdots& G_{i_{d_0+1},i_{d_0+1}}
\end{pmatrix}
\]
are uniformly bounded in operator norm by certain constant $C'$. Hence, by definition \ref{eq:Ri}, $\|\vv{R}_i\| \leq C'$ for any $i$. In particular, for $i=1,\ldots,q-K$,
\begin{align*}
    \vv{R}_i &= \binom{q-1}{d_0}^{-1}\sum_{J=(j_1,\ldots,j_{d_0})\neq i} {\rm det}\begin{pmatrix}
    \vv{\Sigma}_{i} & \vv{\Sigma}_{j_1} &\cdots & \vv{\Sigma}_{j_{d_0}}\\
    G_{j_1,i} & G_{j_1,j_1} & \cdots & G_{j_1,i_{d_0}}\\
    \vdots&\vdots&\ddots&\vdots \\
    G_{j_{d_0},i} & G_{j_{d_0},j_1} &\cdots& G_{j_{d_0},j_{d_0}}
\end{pmatrix}\\
    & = \binom{q-1}{d_0}^{-1}\sum_{l=1}^{\min\{K,d_0\}} \sum_{|J\cap J_0|=l} {\rm det}\begin{pmatrix}
    \vv{\Sigma}_{i} & \vv{\Sigma}_{j_1} &\cdots & \vv{\Sigma}_{j_{d_0}}\\
    G_{j_1,i} & G_{j_1,j_1} & \cdots & G_{j_1,i_{d_0}}\\
    \vdots&\vdots&\ddots&\vdots \\
    G_{j_{d_0},i} & G_{j_{d_0},j_1} &\cdots& G_{j_{d_0},j_{d_0}}.
\end{pmatrix}
\end{align*}
So, we see that for any $i=1,\ldots,q-K$,
\begin{align*}
    \|\vv{R}_i\| & \leq \binom{q-1}{d_0}^{-1}\sum_{l=1}^{\min\{K,d_0\}}  \binom{K}{l}\binom{q-K-1}{d_0-l} C' = o(1)
\end{align*}
Hence,
\begin{align*}
    r_p^{(d_0)} &= \frac{1}{q} \left\{\sum_{j=1}^{q-K} + \sum_{j=q-K+1}^q\right\} \mu(\vv{R}_i,\vv{R}_i|\vv{\Sigma}_i) \\
    & \leq \frac{q-K}{q} o(1) + \frac{K}{q} C' = o(1),
\end{align*}
which proves (\ref{eq:Rapp}).

Now, we see that under $H_1$, it has
    \begin{equation}\label{eq:palt}
    \mathbb{P}\left(\frac{\sqrt{q} p\hat{M}^{(d_0+1)}_p}{2(d_0+1)M_{p,K}^{(d_0)}\beta_p^{1/2}} > z_{\alpha}\right) =\Phi\left(\frac{\sqrt{q} pM_p^{(d_0+1)}}{2(d_0+1)M_{p,K}^{(d_0)} \beta_p^{1/2}} - z_{\alpha}\right) +o(1).
     \end{equation}
We next show that Equation (\ref{eq:ratio}) holds. It suffices for us to focus on the denominator. In fact,
since ${\rm det}(\vv{G}(I)) \geq 0$, we have
\begin{align*}
    M_p^{(d_0+1)} & = \binom{q}{d_0+1}^{-1}\sum_{I = (i_1,\ldots,i_{d_0+1})}{\rm det}(\vv{G}(I)) \\
    & =\binom{q}{d_0+1}^{-1}\sum_{l=1}^{\min\{K,d_0\}}  \sum_{|I\cap J_0| = l}{\rm det}(\vv{G}(I)) \\
    & = \binom{q}{d_0+1}^{-1}\left\{\sum_{j=1}^K \sum_{I: I\cap J_0=\{q-K+j\}}{\rm det}(\vv{G}(I)) + \sum_{l= 2}^{\min\{K,d_0\}}  \sum_{|I\cap J_0| = l}{\rm det}(\vv{G}(I))\right\} \\
    & = \binom{q-K}{d_0}\binom{q}{d_0+1}^{-1} \sum_{j=1}^K M_{p,K}^{(d_0)}(\vv{\Sigma}_{q-K+j})\\
    & + |O\left(\sum_{l=2}^{\min\{K,d_0\}} \binom{K}{l}\binom{q-K}{d_0+1-l}\binom{q}{d_0+1}^{-1}\right)|\\
    & = \binom{q-K}{d_0}\binom{q}{d_0+1}^{-1} M_{p,K}^{(d_0)} \sum_{j=1}^K \langle\vv{\Sigma}_{\perp,q-K+j},\vv{\Sigma}_{\perp,q-K+j}\rangle+O(\frac{d_0K}{q})~~~~(\hbox{by (\ref{eq:detdecomp})})\\
    &=\frac{(d_0+1)}{q} M_{p,K}^{(d_0)}\left\{\sum_{j=1}^K \frac{1}{p}{\rm tr}(\vv{\Sigma}_{\perp,q-K+j}^2)\right\}\left\{\binom{q-K}{d_0}\binom{q-1}{d_0}^{-1}\right\}+|O(\frac{d_0K}{q})|
\end{align*}
    So, (\ref{eq:outbound}) indeed holds.
    
Similar to the  discussion in Appendix \ref{app:E}, it is easy for us to see that $\hat{M}_{p}^{(d_0)} - M_p^{(d_0)} = o_p(1)$. Combining with (\ref{eq:Mapp}), we find
\[
\hat{M}_{p}^{(d_0)} - M_{p,K}^{(d_0)} = \hat{M}_{p}^{(d_0)} - M_p^{(d_0)} + M_p^{(d_0)} - M_{p,K}^{(d_0)} = o_p(1).
\]
Meanwhile, since $\hat{\beta}_p -\beta_p = o_p(1)$, we see that 
\[
(\hat{\sigma}_{p}^{(d_0)})^2 - 4(d_0+1)^2\{M_{p,K}^{(d_0)}\}\beta_p = o_p(1)
\]
and thus
\[
\frac{q^{1/2}p(\hat{M}_p^{(d_0+1)}-M_p^{(d_0+1)})}{\hat{\sigma}_{p}^{(d_0)}}\overset{d.}{\to}\mathcal{N}(0,1)
\]
by the Slutsky lemma. Consequently, by (\ref{eq:palt}), we find
\begin{align*}
    \P_{H_1}\left(\frac{q^{1/2}p\hat{M}_p^{(d_0+1)}}{\hat{\sigma}_{p}^{(d_0)}}>z_\alpha\right) & = \mathbb{P}\left(\frac{\sqrt{q} p\hat{M}^{(d_0+1)}_p}{2(d_0+1)M_{p,K}^{(d_0)}\beta_p^{1/2}} > z_{\alpha}\right) + o(1)\\
    & =\Phi\left(\frac{\sqrt{q} pM_p^{(d_0+1)}}{2(d_0+1)M_{p,K}^{(d_0)} \beta_p^{1/2}} - z_{\alpha}\right) +o(1)\\
    & \geq \Phi\left(\frac{p}{2\sqrt{q}\beta_p^{1/2}}\sum_{j=1}^K \frac{1}{p}{\rm tr}(\vv{\Sigma}_{\perp,q-K+j}^2)  - z_{\alpha}\right)+ o(1).
\end{align*}
The conclusion (\ref{eq:altpw}) immediately follows.

\end{document}